\documentclass[a4paper,10pt]{amsart}

\usepackage[colorlinks=printing]{hyperref}
\usepackage{amssymb,amsmath,graphicx,subfigure}

\newtheorem{theorem}{Theorem}[section]
\newtheorem{lemma}[theorem]{Lemma}

\newtheorem{corollary}[theorem]{Corollary}
\newtheorem{definition}{Definition}[section]

\theoremstyle{remark} \newtheorem{remark}[theorem]{Remark} \theoremstyle{definition} 

\numberwithin{equation}{section}

\newif\ifcomment \commentfalse \def\commentON{\commenttrue}
 \long\outer\def\BC#1\EC{\ifcomment \sloppy \par \# \ldots\dotfill {\em #1} \dotfill \# \par \fi } \commentON

\newcommand{\remove}[1]{}

\newcommand{\eps}{\ensuremath{\varepsilon}}

\newcommand{\R}{{{\mathbb{R}}}}
\newcommand{\N}{{{\mathbb{N}}}}
\newcommand{\Z}{{{\mathbb{Z}}}}
\newcommand{\C}{{{\mathbb{C}}}}
\newcommand{\p}{{{\mathbb{P}}}}

\newcommand{\dx}{\, dx}

\newcommand{\dt}{\, dt}

\newcommand{\dz}{\, dz}

\newenvironment{Definitions}
{%

\begin{enumerate}}%
{\end{enumerate} }

\newcommand{\del}{\partial} \newcommand{\ra}{\rightarrow}

\newcommand{\Levy}{\ensuremath{\mathcal{L}}}

\begin{document}

\title[DG for fractional convection-diffusion equations]
{The discontinuous Galerkin method\\ for fractional degenerate\\ convection-diffusion equations}

\author[{S.~Cifani}]{{Simone Cifani}}
\address[Simone Cifani]{\\ Department of Mathematics\\
Norwegian University of Science and Technology (NTNU)\\
 N-7491 Trondheim, Norway}
 \email[]{simone.cifani\@@math.ntnu.no}
 \urladdr{http://www.math.ntnu.no/\~{}cifani/}

\author[{E.R.~Jakobsen}]{{Espen R. Jakobsen}}
\address[Espen R. Jakobsen]{\\ Department of Mathematics\\ Norwegian University of Science and Technology (NTNU)\\
 N-7491 Trondheim, Norway}
\email[]{erj\@@math.ntnu.no} \urladdr{http://www.math.ntnu.no/\~{}erj/}

\author[{K.H.~Karlsen}]{{Kenneth H. Karlsen}}
\address[Kenneth H. Karlsen]{\\ Centre of Mathematics for
Applications (CMA)\\ Department of Mathematics\\ University of Oslo\\ P.O. Box 1053, Blindern\\ N-0316 Oslo, Norway} \email[]{kennethk\@@math.uio.no}
\urladdr{http://folk.uio.no/kennethk/}


\thanks{This research was supported by the Research Council of Norway (NFR) through the project "Integro-PDEs: Numerical methods, Analysis, and Applications to Finance". The work of K.~H.~Karlsen was also supported through a NFR Outstanding Young Investigator Award.
This article was written as part of the international research program on Nonlinear Partial Differential Equations at the Centre for Advanced Study at the
Norwegian Academy of Science and Letters in Oslo during the academic year 2008--09.}

\begin{abstract}
We propose and study discontinuous Galerkin methods for strongly degenerate convection-diffusion equations perturbed by a fractional diffusion (L\'evy)
operator. We prove various stability estimates along with convergence results toward properly defined (entropy) solutions of linear and nonlinear equations.
Finally, the qualitative behavior of solutions of such equations are illustrated through numerical experiments.
\end{abstract}

\maketitle

\section{Introduction}
We consider degenerate convection-diffusion equations perturbed by a fractional diffusion (L\'{e}vy) operator; more precisely, problems of the form
\begin{align}\label{1}
\begin{cases}
u_{t}+f(u)_{x}=(a(u)u_x)_x+b\Levy[u]&\quad(x,t)\in Q_T=\mathbb{R}\times(0,T),\\
u(x,0)=u_{0}(x)&\quad x\in\mathbb{R},
\end{cases}
\end{align}
where $f,a:\mathbb{R}\rightarrow\mathbb{R}$ ($a\geq0$ and bounded) are Lipschitz continuous functions, $b\geq0$ is a constant, and $\Levy$ is a nonlocal
operator whose singular integral representation reads (cf.~\cite{Landkof,Droniou/Imbert})
\begin{align*}
\Levy[u(x,t)]=c_{\lambda}\int_{|z|>0}\frac{u(x+z,t)-u(x,t)}{|z|^{1+\lambda}}\dz,\ \lambda\in(0,1)\text{ and }c_{\lambda}>0.
\end{align*}
For sake of simplicity, we assume $f(0)=0$. The initial datum $u_0:\R\rightarrow\R$ is chosen in different spaces (cf. Theorems \ref{DDGlin}, \ref{LDGlin} and
\ref{th:existence}) depending on whether the equations are linear or nonlinear.

The operator $\Levy$ is known as the fractional Laplacian (a nonlocal generalization of the Laplace operator) and can also be defined in terms of its Fourier
transform:
\begin{align}\label{fourier}
\widehat{\Levy[u(\cdot,t)]}(\xi)=-|\xi|^{\lambda}\hat{u}(\xi,t).
\end{align}
As pointed out in \cite{Alibaud,Droniou/Imbert,Landkof}, $u(\cdot,t)$ has to be rather smooth with suitable growth at infinity for the quantity $\Levy[u]$ to
be pointwise well defined. However, smooth solutions of \eqref{1} do not exist in general (shocks may develop), and weak entropy solutions have to be
considered, cf.~Definition \ref{def:entropy} and Lemma \ref{lemmanonloc} below.

Nonlocal equations like \eqref{1} appear in different areas of research. For instance, in mathematical finance, option pricing models based on jump processes
(cf.~\cite{Cont/Tankov}) give rise to linear partial differential equations with nonlocal terms. Nonlinear equations appear in dislocation dynamics,
hydrodynamics and molecular biology \cite{Espedal/Karlsen}; applications to semiconductors devices and explosives can also be found \cite{Matalon}. For more
information about the possible applications of such equations we refer the reader to the detailed discussions in \cite{Alibaud},
\cite{Alibaud/Droniou/Vovelle}, and \cite{Droniou}.

Equation \eqref{1} consists of three different terms: nonlinear convection $f(u)_x$, nonlinear diffusion $(a(u)u_x)_x$, and fractional diffusion $\Levy[u]$. It
is expected that the effect of a diffusion operator is that solutions become smoother than the prescribed initial data.  In our case, however, $a$ can be
strongly degenerate (i.e., vanish on intervals of positive length), and hence solutions can exhibit shocks. We refer to \cite{Evje/Karlsen,Espedal/Karlsen} for
the case when $b=0$, and to \cite{Alibaud/Droniou/Vovelle,Cifani/ERJ/KHK} for the case when $\lambda\in(0,1)$ and $a\equiv0$. The issue at stake here is that
the fractional diffusion operator may not be strong enough to prevent solutions of \eqref{1} from developing discontinuities. However, and as expected, in the
linear case ($f(u)=cu$, $a(u)=au$ with $c\in\R$, $a>0$), some regularity can be proved (cf.~Lemma \ref{smoothsol}).

An ample literature is available on numerical methods for computing entropy solutions of degenerate convection-diffusion equations,
cf.~\cite{Cockburn/Shu,Espedal/Karlsen,Evje/Karlsen,Evje/Karlsen2,Holden/Karlsen/Lie1,Holden/Karlsen/Lie2,Karlsen/Risebro1,Karlsen/Risebro2,Hailiang/Jue}. To
the best of our knowledge, there are no works on nonlocal versions of these equations. However, for the special case of fractional conservation laws ($a\equiv
0$) there are a few recent works \cite{Dedner/Rohde,Droniou,Cifani/ERJ/KHK}. Dedner and Rohde \cite{Dedner/Rohde} introduced a general class of difference
methods for equations appearing in radiative hydrodynamics.  Droniou \cite{Droniou} devised a classs difference method for \eqref{1} ($a=0$) and proved
convergence. Cifani \emph{et al.}~\cite{Cifani/ERJ/KHK} applied the discontinuous Galerkin method to \eqref{1} ($a=0$) and proved error estimates. Finally, let
us mention that the discontinuous Galerkin method has also been used to numerically solve nonlinear convection problems appended with possibly nonlocal
dissipative terms in \cite{Jenny,Jenny/Rodhe}.

The discontinuous Galerkin (DG hereafter) method is a well established method for approximating solutions of convection \cite{Cockburn} and
convection-diffusion equations \cite{Cockburn/Shu,Hailiang/Jue}. To obtain a DG approximation of a nonlinear equation, one has to pass to the weak formulation,
do integration by parts, and replace the nonlinearities with suitable numerical fluxes (fluxes which enforce numerical stability and convergence). Available DG
methods for convection-diffusion equations are the local DG (LDG hereafter) \cite{Cockburn/Shu} and the direct DG (DDG hereafter) \cite{Hailiang/Jue}. In the
LDG method, the convection-diffusion equation is rewritten as a first order system and then approximated by the DG method for conservation laws. In the DDG
method, the DG method is applied directly to the convection-diffusion equation after a suitable numerical flux has been derived for the diffusion term.

This paper is a continuation of our previous work on DG methods for fractional conservation laws \cite{Cifani/ERJ/KHK}. We devise and study DDG and LDG
approximations of \eqref{1}, we prove that both approximations are $L^2$-stable and, whenever linear equations are considered, high-order accurate. In the
nonlinear case, we work with an entropy formulation for \eqref{1} which generalizes the one in \cite{Wu/Yin,Evje/Karlsen}, and we show that the DDG method
converges toward an entropy solution when piecewise constant elements are used. To do so, we extend the results in \cite{Evje/Karlsen} to our nonlocal setting.
Finally, we present numerical experiments shedding some light on the qualitative behavior of solutions of fractional, strongly degenerate convection-diffusion
equations.

\section{A semi-discrete method}
Let us choose a spatial grid $x_{i}=i\Delta x$ ($\Delta x>0$, $i\in\mathbb{Z}$), and label $I_{i}=(x_{i},x_{i+1})$. We denote by $P^{k}(I_{i})$ the space of
all polynomials of degree at most $k$ with support on $I_i$, and let
\begin{align*}
V^k=\{v:v|_{I_i}\in P^k(I_i),\ i\in\Z\}.
\end{align*}
Let us introduce the Legendre polynomials $\{\varphi_{0,i},\varphi_{1,i},\ldots,\varphi_{k,i}\}$, where $\varphi_{j,i}\in P^{j}(I_{i})$. Each function in
$P^{k}(I_{i})$ can be written as a linear combination of these polynomials.

We recall the following well known properties of the Legendre polynomials: for all $i\in\mathbb{Z}$,
\begin{align*}
\int_{I_{i}}\varphi_{p,i}\varphi_{q,i}\dx&=\left\{
\begin{array}{cl}
\frac{\Delta x}{2q+1}&\text{for }p=q\\
0&\text{otherwise}
\end{array}
\right.,\ \varphi_{p,i}(x_{i+1}^{-})=1\text{ and }\varphi_{p,i}(x_{i}^{+})=(-1)^{p},
\end{align*}
where $\varphi(x_i^{\pm})=\lim_{s\rightarrow x_i^{\pm}}\varphi(s)$.

The following fractional Sobolev space is also needed in what follows (see, e.g., \cite{Abels/Kassmann} or \cite[Section 6]{Folland}):
\begin{align*}
\|u\|_{H^{\lambda/2}(\mathbb{R})}^{2}=\|u\|_{L^{2}(\mathbb{R})}^{2} +|u|_{H^{\lambda/2}(\mathbb{R})}^2,
\end{align*}
with semi-norm $|u|_{H^{\lambda/2}(\mathbb{R})}^2= \int_\R\int_\R\frac{(u(z)-u(x))^2}{|z-x|^{1+\lambda}}\dz\dx$. Finally, let us introduce the operators
\begin{align*}
[p(x_i)]=p(x_i^+)-p(x_i^-),\quad \overline{p(x_i)}=\frac{1}{2}(p(x_i^+)+p(x_i^-)).
\end{align*}
 From now on we split our exposition into two parts, one dedicated to
 the DDG method and another one dedicated to the LDG method.

\subsection{DDG method}
Let us multiply \eqref{1} by an arbitrary $v\in P^k(I_i)$, integrate over $I_i$, and use integration by parts, to arrive at
\begin{align}\label{Ds1}
\begin{split}
\int_{I_{i}}u_{t}v\ -&\int_{I_{i}}f(u)v_{x}+f(u_{i+1})v_{i+1}^{-}-f(u_{i})v_{i}^{+}\\
&+\int_{I_i}a(u)u_xv_x-h(u_{i+1},u_{x,i+1})v_{i+1}^-+h(u_{i},u_{x,i})v_{i}^+ =b\int_{I_{i}}\Levy[u]v,
\end{split}
\end{align}
where $f(u_i)=f(u(x_i))$, $h(u,u_x)=a(u)u_x$ and $(u_i,u_{x,i})=(u(x_i),u_x(x_i))$. Let us introduce the Lipschitz continuous E-flux (a consistent and monotone
flux),
\begin{align}\label{eflux}
\hat{f}(u_i)=\hat{f}(u(x_i^-),u(x_i^+)).
\end{align}
Note that since $\hat{f}$ is consistent ($\hat{f}(u,u)=f(u)$) and monotone (increasing w.r.t.~ its first variable and decreasing w.r.t~ its second variable),
\begin{align}\label{eeeflux}
\int_{u_i^-}^{u_i^+}\big[f(x)-\hat{f}(u_i^-,u_i^+)\big]\dx\geq0.
\end{align}
Following Jue and Liu \cite{Hailiang/Jue}, let us also introduce the flux
\begin{align*}
\hat{h}(u_i)&=\hat{h}(u(x_i^-),\ldots,\partial_x^ku(x_i^-),u(x_i^+),\ldots,\partial_x^ku(x_i^+))\\
&=\beta_0\frac{[A(u_i)]}{\Delta x}+\overline{A(u_i)_x} +\sum_{m=1}^{\lfloor k/2\rfloor}\beta_m\Delta x^{2m-1}[\partial_x^{2m}A(u_i)],
\end{align*}
where $A(u)=\int^{u}a$ and the weights $\{\beta_0,\ldots,\beta_{\lfloor k/2\rfloor}\}$ fulfill the following admissibility condition: there exist
$\gamma\in(0,1)$ and $\alpha\geq0$ such that
\begin{align}\label{admiss}
\sum_{i\in\Z}\hat{h}(u_i)[u_i]\geq\alpha \sum_{i\in\Z}\frac{[A(u_i)]}{\Delta x}[u_i] -\gamma\sum_{i\in\Z}\int_{I_i}a(u)(u_x)^2.
\end{align}
Note that the numerical flux $\hat h$ is an approximation of $A(u_i)_x=a(u(x_i))u_x(x_i)$ involving the average $\overline{A(u_i)_x}$ and the jumps of even
order derivatives of $A(u_i)$ up to $m=k/2$. For example, if $k=0$ and $\beta_0=1$, then
$$
\hat{h}(u_i)=\frac{1}{\Delta x} [A(u_i)]=\frac{A(u(x_i^+))-A(u(x_i^-))}{\Delta x},$$ and this function satisfies condition \eqref{admiss}. In this case
($k=0$),
$$\overline{A(u_i)_x}=\overline{a(u(x_i))\partial_xu(x_i)}
=\frac{1}{2}\Big(a(u(x_i^+))\partial_xu(x_i^+)+a(u(x_i^-))\partial_xu(x_i^-)\Big)=0.$$ When $k\geq2$, some extra differentiability on $a$ is required. For
example, with $k=2$,
\begin{align*}
\sum_{m=1}^{\lfloor k/2\rfloor}\beta_m\Delta x^{2m-1}[\partial_x^{2m}A(u_i)]&=\beta_1\Delta x[\partial_x^{2}A(u_i)]=\beta_1\Delta
x[a'(u_i)(\partial_xu_i)^2+a(u_i)\partial_x^{2}u_i].
\end{align*}
We see that the flux $\hat h$ is locally Lipschitz if $a$ is sufficently regular, and that $\hat h(0)=0$ for all $k$. Let us rewrite \eqref{Ds1} as
\begin{align}\label{Ds2}
\begin{split}
\int_{I_{i}}u_{t}v\ -&\int_{I_{i}}f(u)v_{x}+\hat{f}(u_{i+1})v_{i+1}^{-}-\hat{f}(u_{i})v_{i}^{+}\\
&+\int_{I_i}a(u)u_xv_x-\hat{h}(u_{i+1})v_{i+1}^-+\hat{h}(u_{i})v_{i}^+=b\int_{I_{i}}\Levy[u]v,
\end{split}
\end{align}
and use the initial condition
\begin{align}\label{initial_condition}
\begin{split}
\int_{I_{i}}u(x,0)v(x)\dx=\int_{I_{i}}u_0(x)v(x)\dx.
\end{split}
\end{align}
The DDG method consists of finding functions $\hat{u}:Q_T\rightarrow\mathbb{R}$, $\hat{u}(\cdot,t)\in V^k$, and
\begin{align}\label{udef}
\hat{u}(x,t)=\sum_{i\in\mathbb{Z}}\sum_{p=0}^{k}U_{p,i}(t)\varphi_{p,i}(x),
\end{align}
which satisfy \eqref{Ds2}-\eqref{initial_condition} for all $v\in P^{k}(I_i)$, $i\in\mathbb{Z}$.

\subsection{LDG method}
Let us write $a(u)u_x=\sqrt{a(u)}g(u)_x$, where $g(u)=\int^{u}\sqrt{a}$, and turn equation \eqref{1} into the following system of equations
\begin{align}\label{prob}
\left\{
\begin{array}{ll}
u_t+(f(u)-\sqrt{a(u)}q)_x=b\Levy[u],\\
q-g(u)_x=0.
\end{array}
\right.
\end{align}
Let us introduce the notation $\mathbf{w}=(u,q)'$ (here $'$ denotes the transpose), and write
\begin{align*}
\mathbf{h}(\mathbf{w})=\mathbf{h}(u,q)=\left(
\begin{array}{c}
h_u(\mathbf{w})\\
h_q(u)
\end{array}
\right)=\left(
\begin{array}{c}
f(u)-\sqrt{a(u)}q\\
-g(u)
\end{array}
\right).
\end{align*}
Let us multiply each equation in \eqref{prob} by arbitrary $v_u,v_q\in P^{k}(I_{i})$, integrate over the interval $I_{i}$, and use integration by parts, to
arrive at
\begin{align*}
&\int_{I_{i}}\partial_tuv_{u}-\int_{I_{i}}h_u(\mathbf{w})\partial_{x}v_u +h_u(\mathbf{w}_{i+1})v_{u,i+1}^{-}-h_u(\mathbf{w}_{i})v_{u,i}^{+}
=b\int_{I_{i}}\Levy[u]v_u,\\
&\int_{I_{i}}qv_q-\int_{I_{i}}h_q(u)\partial_xv_{q}+h_q(u_{i+1}) v_{q,i+1}^{-}-h_q(u_{i})v_{q,i}^{+}=0,
\end{align*}
where $h_u(\mathbf{w}_{i})=h_u(u_{i},q_i)$, $u_i=u(x_i)$, $q_i=q(x_i)$, $v_{u,i}^-=v_u(x_i^-)$ and $v_{u,i}^+=v_u(x_i^+)$. Following Cockburn and Shu
~\cite{Cockburn/Shu}, we introduce the numerical flux
\begin{align}\label{Lflux}
\hat{\mathbf{h}}(\mathbf{w}_i^-,\mathbf{w}_i^+)=\left(
\begin{array}{c}
\hat{h}_u(\mathbf{w}_i^-,\mathbf{w}_i^+)\\
\hat{h}_q(u_i^-,u_i^+)
\end{array}
\right)=\left(
\begin{array}{c}
\frac{[F(u_i)]}{[u_i]}-\frac{[g(u_i)]}{[u_i]}\overline{q_i}\\
-\overline{g(u_i)}
\end{array}
\right)-\C[\mathbf{w}_i],
\end{align}
where $F(u)=\int^{u}f$, $\C=\left(
\begin{array}{cc}
c_{11} & c_{12}\\
-c_{12} & 0
\end{array}
\right)$,
\begin{align*}
\begin{split}
c_{11}=\frac{1}{[u_i]}\left(\frac{[F(u_i)]}{[u_i]}-\hat{f}(u_i^-,u_i^+)\right),
\end{split}
\end{align*}
$c_{12}=c_{12}(\mathbf{w}_i^-,\mathbf{w}_i^+)$ is Lipschitz continuous in all its variables, and $c_{12}=0$ whenever $a=0$ or
$\mathbf{w}_i^-,\mathbf{w}_i^+=0$. Note that $c_{11}\geq0$ since $\hat{f}$ is an E-flux and, thus, the matrix $\C$ is semipositive definite.

The LDG method consists of finding $\tilde{\mathbf{w}}=(\tilde{u},\tilde{q})'$, where
\begin{align*}
\tilde{u}(x,t)=\sum_{i\in\mathbb{Z}}\sum_{p=0}^{k}U_{p,i}(t)\varphi_{p,i}(x)\quad\text{and}\quad
\tilde{q}(x,t)=\sum_{i\in\mathbb{Z}}\sum_{p=0}^{k}Q_{p,i}(t)\varphi_{p,i}(x)
\end{align*}
are functions satisfying
\begin{align}\label{systdisc}
\begin{split}
&\int_{I_{i}}\partial_tuv_{u}-\int_{I_{i}}h_u(\mathbf{w})\partial_{x}v_u+\hat{h}_u(\mathbf{w}_{i+1})v_{u,i+1}^{-}
-\hat{h}_u(\mathbf{w}_{i})v_{u,i}^{+}=b\int_{I_{i}}\Levy[u]v_u,\\
&\int_{I_{i}}qv_q-\int_{I_{i}}h_q(u)\partial_xv_{q}+\hat{h}_q(u_{i+1})v_{q,i+1}^{-}-\hat{h}_q(u_{i})v_{q,i}^{+}=0,
\end{split}
\end{align}
for all $v_u,v_q\in P^{k}(I_i)$, $i\in\mathbb{Z}$, and initial conditions for $u$ and $q$ given by \eqref{initial_condition}.

\section{$L^2$-stability for nonlinear equations}
We will show that in the semidiscrete case (no time discretization) both the DDG and LDG methods are $L^2$-stable, for linear and nonlinear equations.

In this section and the subsequent one, we assume the existence of solutions $\hat{u}$ and $\tilde{\mathbf{w}}=(\tilde{u},\tilde{q})'$ of the DDG and LDG
methods \eqref{Ds2} and \eqref{systdisc}, respectively, satisfying $\hat{u},\tilde{u},\tilde{q}\in C^1([0,T];V^k\cap L^{2}(\mathbb{R}))$, in which case the
integrals containing the nonlocal operator $\Levy[\cdot]$ are all well defined. Indeed, by Lemma \ref{Hs}, $V^k\cap L^{2}(\mathbb{R})\subseteq
H^{\lambda/2}(\mathbb{R})$, and hence all integrals of the form
$$
\int_{\mathbb{R}}\varphi_1\,\Levy[\varphi_2]\quad\text{for$ \quad\varphi_1,\varphi_2\in V^k\cap L^{2}(\mathbb{R})$,}
$$
can be interpreted as the pairing between $\varphi_1\in H^{\lambda/2}(\mathbb{R})$ and $\Levy[\varphi_2]\in H^{-\lambda/2}(\mathbb{R})$. Here
$H^{-\lambda/2}(\mathbb{R})$ is the dual space of $H^{\lambda/2}(\mathbb{R})$, and $\Levy[\varphi]\in H^{-\lambda/2}(\mathbb{R})$ whenever $\varphi\in
H^{\lambda/2}(\mathbb{R})$ (cf.~Corollary A.3 and proof in \cite{Cifani/ERJ/KHK}).

\begin{remark}
The existence and uniqueness of solutions in $C^1([0,T];V^k\cap L^{2}(\mathbb{R}))$ can be proved using the Picard-Cauchy-Lipschitz theorem. The argument
outlined in \cite[Section 3]{Cifani/ERJ/KHK}, can be adapted to the current setting since all numerical fluxes are (locally) Lipschitz (cf. \cite{Cockburn/Shu}
for the LDG case). For the DDG method with $k>2$, additional differentiability on $a$ is needed for this proof to work.
\end{remark}

\subsection{DDG method}
Let us sum over all $i\in\Z$ in \eqref{Ds2}, integrate over $t\in(0,T)$, and introduce the functional
\begin{align}\label{functionalM}
\begin{split}
M_{DDG}[u,v]=\int_0^T&\int_{\mathbb{R}}u_{t}v
-\int_0^T\sum_{i\in\mathbb{Z}}\Big[\hat{f}(u_{i})[v_{i}]+\int_{I_{i}}f(u)v_{x}\Big]\\
&+\int_0^T\sum_{i\in\mathbb{Z}}\Big[\hat{h}(u_i)[v_i] +\int_{I_i}a(u)u_xv_x\Big]-b\int_0^T\int_{\mathbb{R}}\Levy[u]v.
\end{split}
\end{align}
Let us define
\begin{align*}
\Gamma_T[u]=(1-\gamma)\int_0^T\sum_{i\in\Z}\int_{I_i}a(u)(u_x)^2 +\alpha\int_0^T \sum_{i\in\Z}\frac{[A(u_i)]}{\Delta x}[u_i],
\end{align*}
where $\gamma\in(0,1)$ and $\alpha>0$. Note that $\Gamma_T\geq0$ since $a\geq0$ and, using the Taylor's formula, $[A(u_i)][u_i]=a(\xi_i)[u_i]^2\geq0$ where and
$\xi_i\in[u(x_i^-),u(x_i^+)]$, $i\in\Z$.
\begin{theorem}\label{stab1}
\emph{(Stability)} Let $\hat{u}$ be a solution of \eqref{Ds2} such that both $\hat u,A(\hat u)$ and their first $k$ derivatives are sufficiently integrable.
Then
\begin{align*}
\|\hat{u}(\cdot,T)\|_{L^{2}(\mathbb{R})}^2+2\Gamma_{T}[\hat{u}] +bc_\lambda\int_0^T|\hat{u}(\cdot,t)|^2_{H^{\lambda/2}(\mathbb{R})}\dt
\leq\|u_{0}\|_{L^{2}(\mathbb{R})}^2.
\end{align*}
\end{theorem}

\begin{remark}
Since $\tilde u\in C^1([0,T];V^k\cap L^{2}(\mathbb{R}))$ and $f(0)=0$, all terms in \eqref{main} below are well defined -- except for
\begin{align*}
\int_0^T\Big[\sum_{i\in\Z}\hat{h}(\hat u_i)[\hat{u}_i] +\int_{I_i}a(\hat{u})(\hat{u}_x)^2\Big].
\end{align*}
When $k\geq2$, additional integrability of $\hat u,A(\hat u)$, and their first $k$ derivatives, is required in order to give meaning to the $\hat h$-term.
\end{remark}

\begin{proof}
By construction, $M_{DDG}[\hat{u},v]=0$ for all $v\in V^k\cap L^{2}(\R)$. If we set $v=\hat{u}$, we obtain
\begin{align}\label{main}
\begin{split}
\int_0^{T}\int_{\mathbb{R}}\hat{u}_{t}\hat{u}&-\int_0^{T}\sum_{i\in\mathbb{Z}}\Big[\hat{f}(\hat{u}_{i})
[\hat{u}_{i}]+\int_{I_{i}}f(\hat{u})\hat{u}_{x}\Big]\\
&+\int_0^T\sum_{i\in\Z}\Big[\hat{h}(\hat u_i)[\hat{u}_i] +\int_{I_i}a(\hat{u})(\hat{u}_x)^2\Big]-b\int_0^{T}\int_{\mathbb{R}}\Levy[\hat{u}]\hat{u}=0.
\end{split}
\end{align}
Next, as a direct consequence of \eqref{eeeflux} and a change of variables, we see that
\begin{align}\label{e1}
\int_0^{T}\sum_{i\in\mathbb{Z}}\Big[\hat{f}(\hat{u}_{i})[\hat{u}_{i}] +\int_{I_{i}}f(\hat{u})\hat{u}_{x}\Big]\leq0.
\end{align}
Since $\hat{h}$ satisfies the expression \eqref{admiss},
\begin{align}\label{e2}
\int_0^{T}\sum_{i\in\Z}\hat{h}_i(\hat u_i)[\hat{u}_i]\geq \alpha\int_0^T\sum_{i\in\Z}\frac{[A(\hat{u}_i)]}{\Delta x}[\hat{u}_i]
-\gamma\int_0^T\sum_{i\in\Z}\int_{I_i}a(\hat{u})(\hat{u}_x)^2.
\end{align}
Finally, using Lemma \ref{lemmanonloc},
\begin{align}\label{e3}
\int_{\mathbb{R}}\Levy[\hat{u}]\hat{u}= -\frac{c_\lambda}{2}|\hat{u}|_{H^{\lambda/2}(\mathbb{R})}^2.
\end{align}
We conclude by inserting \eqref{e1}, \eqref{e2}, and \eqref{e3} into \eqref{main}.
\end{proof}

\subsection{LDG method}
By summing over all $i\in\Z$, we can rewrite \eqref{systdisc} as
\begin{align*}
&\int_{\R}\partial_{t}uv_u-\sum_{i\in\Z}\bigg(\hat{h}_u(\mathbf{w}_{i})[v_{u,i}]
+\int_{I_{i}}h_u(\mathbf{w})\partial_{x}v_u\bigg)=b\int_{\R}\Levy[u]v_u,\\
&\int_{\R}q v_q-\sum_{i\in\Z}\bigg(\hat{h}_q(u_{i})[v_{q,i}] +\int_{I_{i}}h_{q}(u)\partial_xv_q\bigg)=0.
\end{align*}
We add the two equations and integrate over $t\in(0,T)$ to find $M_{LDG}[\mathbf{w},\mathbf{v}]=0$ for
\begin{align}\label{functionalB}
\begin{split}
M_{LDG}[\mathbf{w},\mathbf{v}]=&\int_{0}^{T}\int_{\mathbb{R}}u_tv_u+\int_0^T\int_{\R}qv_q\\
&-\int_0^T\sum_{i\in\Z}\bigg(\hat{\mathbf{h}}(\mathbf{w}_i)'[\mathbf{v}_i]
+\int_{I_i}\mathbf{h}(\mathbf{w})'\partial_x\mathbf{v}\bigg)-b\int_0^T\int_{\R}\Levy[u]v_u,
\end{split}
\end{align}
where $\hat{\mathbf{h}}(\mathbf{w}_i)= (\hat{h}_u(\mathbf{w}_{i}),\hat{h}_q(u_{i}))'$, $\mathbf{v}=(v_u,v_q)'$ and $\mathbf{v}_i=(v_{u,i},v_{q,i})'$. Moreover,
let (remember that, as noted earlier, the matrix $\C$ is semipositive definite)
\begin{align*}
\Theta_{T}[\mathbf{w}]=\int_0^T\sum_{i\in\Z}[\mathbf{w}_i]'\C[\mathbf{w}_i]\ (\geq0).
\end{align*}

\begin{theorem}
\emph{(Stability)} If $\tilde{\mathbf{w}}=(\tilde{u},\tilde{q})'$ is a $C^1([0,T];(V^k\cap L^2)^2 )$ solution of \eqref{systdisc}, then
\begin{align*}
\|\tilde{u}(\cdot,T)\|_{L^{2}(\mathbb{R})}^2+2\|\tilde{q}\|_{L^2(Q_T)}^2 +2\Theta_{T}(\tilde{\mathbf{w}})
+bc_\lambda\int_0^T|\tilde{u}(\cdot,t)|^2_{H^{\lambda/2}(\mathbb{R})}\dt \leq\|u_{0}\|_{L^{2}(\mathbb{R})}^2.
\end{align*}
\end{theorem}
Here, as opposed to Theorem \ref{stab1}, no further integrability of the first $k$ derivatives of the numerical solution
$\tilde{\mathbf{w}}=(\tilde{u},\tilde{q})'$ is needed. The reason is that the numerical flux $\hat{\mathbf{h}}$ has been built without the use of derivatives
of $\tilde{\mathbf{w}}=(\tilde{u},\tilde{q})'$. Each term in expression \eqref{main2} below is well defined thanks to \eqref{e4}, the fact that $f(0)=0$ (which
implies that $c_{11}(0)=0$), $c_{12}(0)=0$, and $\tilde{u},\tilde{q}\in C^1([0,T];V^k\cap L^{2}(\mathbb{R}))$.

\begin{proof}
By construction, $M_{LDG}(\hat{\mathbf{w}},\mathbf{v})=0$ for all $\mathbf{v}=(v_{u},v_{q})'$, $v_u,v_q\in V^{k}\cap L^{2}(\R)$. We set
$\mathbf{v}=\hat{\mathbf{w}}$ and find that
\begin{align}\label{main2}
\begin{split}
\int_0^T\int_{\R}\tilde{u}_t\tilde{u}+\int_0^T\int_{\R}\tilde{q}^2 -\int_0^T\sum_{i\in\Z}\bigg( \hat{\mathbf{h}}(\tilde{\mathbf{w}}_i)'[\tilde{\mathbf{w}}_i]
+\int_{I_{i}}\mathbf{h}(\tilde{\mathbf{w}})'
\partial_x\tilde{\mathbf{w}}\bigg)&\\-b\int_0^T
\int_{\R}\Levy[\tilde{u}]\tilde{u}&=0.
\end{split}
\end{align}
Here we also used the fact that
\begin{align}\label{e4}
\begin{split}
-\int_0^T\sum_{i\in\Z}\bigg(\hat{\mathbf{h}}(\tilde{\mathbf{w}}_i)'[\tilde{\mathbf{w}}_i]
+\int_{I_{i}}\mathbf{h}(\tilde{\mathbf{w}})'\partial_x\tilde{\mathbf{w}}\bigg) =\int_0^T\sum_{i\in\Z}[\tilde{\mathbf{w}}_i] \C[\tilde{\mathbf{w}}_i],
\end{split}
\end{align}
see \cite{Cockburn/Shu} for a proof. To conclude, insert \eqref{e4} and \eqref{e3} into \eqref{main2}.
\end{proof}

\section{High-order convergence for linear equations}
In this section we consider the linear problem
\begin{align}\label{linear}
\begin{cases}
u_{t}+cu_{x}=u_{xx}+b\Levy[u]&(x,t)\in Q_T,\\
u(x,0)=u_0(x)&x\in\R,
\end{cases}
\end{align}
with the aim of proving that the DDG and LDG methods converge to a regular solution of \eqref{linear} with high-order accuracy.

\begin{lemma}\label{smoothsol}
Let $u_{0}\in H^{k+1}(\mathbb{R})$, with $k\geq0$. There exists a
unique function $u\in H^{k+1}_{\mathrm{par}}(Q_T)$ solving \eqref{linear}, where
$$ H^{k+1}_{\mathrm{par}}(Q_T):=\Big\{\phi\in L^2(Q_T):
\|\del_t^m\del_x^{r}u\|_{L^2(Q_T)}<\infty \text{ for all }
0\leq r+2m\leq k+1\Big\}.$$
Moreover, $\|u(\cdot,t)\|_{H^{k+1}(\mathbb{R})}
\leq\|u_{0}\|_{H^{k+1}(\mathbb{R})}.$
\end{lemma}

\begin{proof}
Since the equation is linear, we can pass to the Fourier space. In view of \eqref{fourier}, the Fourier transform of \eqref{linear} is $\hat{u}_{t}+i\xi
c\hat{u}= -\xi^2\hat{u}-b|\xi|^{\lambda}\hat{u}$. It follows that
\begin{align*}
\hat{u}(\xi,t)=\hat{u}_{0}(\xi)e^{-(i\xi c+\xi^2+b|\xi|^{\lambda})t}.
\end{align*}
By the properties of the Fourier transform, the above expression implies the existence of a unique $L^{2}$-stable weak solution of \eqref{linear}. The
$L^{2}$-stability for higher derivatives can be obtained by iteration as follows: take the derivative of \eqref{linear}, use the Fourier transform to get
stability, and iterate up to the $k$th derivative. Regularity in time
follows from the regularity in space since equation \eqref{linear}
implies that $\del_t^ku=(-c\del_x+\del_x^2+b\Levy)^ku$.
\end{proof}

In the following two theorems we obtain $L^2$-type error estimates for
the DDG and LDG methods in the case that equation \eqref{linear} has
$H^{k+1}_{\mathrm{par}}$-regular solutions. (Note that the time regularity does
not play any role here). To do so, we combine estimates for the local terms derived in \cite{Cockburn/Shu,Hailiang/Jue} with estimates for the nonlocal term derived by the
authors in \cite{Cifani/ERJ/KHK}. In \cite{Cockburn} it was observed that most relevant numerical $\hat f$ fluxes reduce to
\begin{align*}
\hat{f}(u_i^-,u_i^+)&=c\overline{u_i}-|c|\frac{[u_i]}{2}
\end{align*}
in the linear case. In this section we only consider this $\hat f$ flux.

\subsection{DDG method}

\begin{theorem}\label{DDGlin}
\emph{(Convergence)} Let $u\in H^{k+1}_{\mathrm{par}}(Q_T)$, $k\geq0$, be a solution of \eqref{linear} and $\hat{u}\in C^1([0,T];V^k\cap L^2(\R))$ be a solution of
\eqref{Ds2}. With $e=u-\hat{u}$,
\begin{align*}
\begin{split}
\int_\R e^2(x,T) & +\frac{|c|}{2}\int_0^T \sum_{i\in\Z}[e_i]^2+(1-\gamma) \int_0^T\int_{\R}(e_x)^2 +\alpha\int_0^{T}
\sum_{i\in\Z}\frac{[e_i]^2}{\Delta x}\\
& +bc_\lambda\int_0^T|e|^2_{H^{\lambda/2}(\R)}
 =\mathcal{O}(1)\Delta x^{2k}.
\end{split}
\end{align*}
\end{theorem}
\begin{remark}
The error $\mathcal{O}(1)\Delta x^{2k}$ is due to the diffusion term $u_{xx}$. The errors from the convection term $cu_x$ and the fractional diffusion term
$b\Levy[u]$ are of the form $\mathcal{O}(1)\Delta x^{2k+1}$ and $\mathcal{O}(1)\Delta x^{2k+2-\lambda}$ respectively.
\end{remark}

\begin{proof}
Let us set
\begin{align*}
M_{a}[u,v]&=\int_0^T\int_{\mathbb{R}}u_{t}v+\int_0^T\int_{\R}u_xv_x
+\int_0^{T}\sum_{i\in\Z}\hat h(u_i)[v_i],\\
M_{f}[u,v]&=-\int_0^T\sum_{i\in\mathbb{Z}}\Big[\hat{f}(u_{i})[v_{i}]
+\int_{I_{i}}cuv_{x}\Big],\\
M_\Levy[u,v]&=-b\int_0^T\int_{\mathbb{R}}\Levy[u]v.
\end{align*}
With this notation in hand, we can write \eqref{functionalM} as
\begin{align*}
\begin{split}
M_{DDG}[u,v]=M_{a}[u,v]+M_{f}[u,v]+M_\Levy[u,v].
\end{split}
\end{align*}
Let $\p e$ be the $L^{2}$-projection of $e$ into $V^k$, i.e., $\p e$ is the $V^k\cap L^2(\mathbb{R})$ function satisfying
\begin{align*}
\int_{I_{i}}\big(\p e(x)-e(x)\big)\varphi_{ji}(x)\dx=0 \text{ for all $i\in\Z$ and $j=\{0,\ldots,k$\}.}
\end{align*}
Note that $\p e\in H^{\lambda/2}(\mathbb{R})$ since $V^k\cap L^2(\mathbb{R})\subset H^{\lambda/2}(\mathbb{R})$ by Lemma \ref{Hs}. For all $v\in V^k\cap
L^2(\mathbb{R})$, we have $M_{DDG}[\hat{u},v]=0$ since $\hat{u}$ is a DDG solution of \eqref{linear}, while $M_{DDG}[u,v]=0$ since $u$ is a continuous (by
Sobolev imbedding) solution of \eqref{1} and hence a solution of \eqref{linear}. Thus $M_{DDG}[e,v]=0$, and by bilinearity ($\hat h$ is linear
 since $a\equiv1$),
\begin{align}\label{iii}
M_{DDG}[\p e,\p e]=M_{DDG}[\p e-e,\p e].
\end{align}
One can proceed as in \cite{Hailiang/Jue} (in that paper, combine
 the last inequality of the proof of Lemma 3.3 with Lemma 3.2 and
 (3.5)) to obtain
\begin{align}\label{i1}
\begin{split}
M_{a}[\p e-e,\p e]=\frac{1}{2}\int_0^T\int_{\R}(\p e_x)^2 +\frac{1}{2}\int_0^{T}\sum_{i\in\Z}\hat h(\p e_i)[\p e_i]+\mathcal{O}(1)\Delta x^{2k}.
\end{split}
\end{align}
Moreover, proceeding as in \cite[Lemma 2.17]{Cockburn},
\begin{align}\label{i2}
\begin{split}
M_{f}[\p e-e,\p e]=\frac{|c|}{4}\int_0^T \sum_{i\in\Z}[\p e_i]^2+\mathcal{O}(1)\Delta x^{2k+1}.
\end{split}
\end{align}
As shown by the authors in \cite{Cifani/ERJ/KHK},
\begin{align}\label{i3}
\begin{split}
&M_\Levy[\p e-e,\p e]-M_\Levy[\p e,\p e]=b\int_{0}^{T}\int_{\mathbb{R}}\Levy[e]\p e\\
&=\frac{b}{2}\int_{0}^{T}\int_{\mathbb{R}}\Levy[\p e]\p e +\frac{b}{2}\int_{0}^{T}\int_{\mathbb{R}}\Levy[e]e
-\frac{b}{2}\int_{0}^{T}\int_{\mathbb{R}}\Levy[e-\p e](e-\p e)\\
&\leq-\frac{bc_\lambda}{4}\int_{0}^{T}|\p e|_{H^{\lambda/2}(\mathbb{R})}^{2} -\frac{bc_\lambda}{4}\int_{0}^{T}|e|_{H^{\lambda/2}(\mathbb{R})}^{2}
+\frac{bc_\lambda}{4}\int_{0}^{T}\|e-\p e\|_{H^{\lambda/2}(\mathbb{R})}^{2},
\end{split}
\end{align}
where $M_\Levy[\p e,\p e]=\frac{bc_\lambda}{2}\int_{0}^{T}|\p e|_{H^{\lambda/2}(\mathbb{R})}^{2}$ (Lemma \ref{lemmanonloc}) and
\begin{align}\label{aaa}
\|e-\p e\|_{H^{\lambda/2}(\mathbb{R})}^{2}\leq\mathcal{O}(1)\Delta x^{2k+2-\lambda}.
\end{align}
By \eqref{functionalM}, Lemma \ref{lemmanonloc}, and the definition of $\hat f$,
\begin{align*}
\begin{split}
M_{DDG}[\p e,\p e]=&\int_\R (\p e^2)_t+\frac{|c|}{2}\int_0^T\sum_{i\in\Z}[\p e_i]^2 +\int_0^T\int_{\R}(\p e_x)^2\\&+\int_0^{T}\sum_{i\in\Z}\hat h(\p e_i)[\p
e_i]+\frac{bc_\lambda}{2}\int_0^T|\p e|^2_{H^{\lambda/2}(\R)}.
\end{split}
\end{align*}
Inserting this equation along with \eqref{i1}, \eqref{i2}, and \eqref{i3} into \eqref{iii} then shows that
\begin{align*}
\begin{split}
\int_0^T\int_\R (\p e^2)_t+\frac{|c|}{4}\int_0^T\sum_{i\in\Z}[\p e_i]^2 +\frac{1}{2}\int_0^T\int_{\R}(\p e_x)^2
+\frac{1}{2}\int_0^{T}\sum_{i\in\Z}\hat h(\p e_i)[\p e_i]&\\
+\frac{bc_\lambda}{4}\int_0^T|\p e|^2_{H^{\lambda/2}(\R)} +\frac{bc_\lambda}{4}\int_0^T| e|^2_{H^{\lambda/2}(\R)}=\mathcal{O}(1)\Delta x^{2k},&
\end{split}
\end{align*}
and, using the admissibility condition \eqref{admiss},
\begin{align*}
\begin{split}
\int_0^T\int_\R (\p e^2)_t+\frac{|c|}{4}\int_0^T\sum_{i\in\Z}[\p e_i]^2 +\frac{1-\gamma}{2}\int_0^T\int_{\R}(\p e_x)^2
+\frac{\alpha}{2}\int_0^{T}\sum_{i\in\Z}\frac{[\p e_i]^2}{\Delta x}&\\
+\frac{bc_\lambda}{4}\int_0^T|\p e|^2_{H^{\lambda/2}(\R)} +\frac{bc_\lambda}{4}\int_0^T| e|^2_{H^{\lambda/2}(\R)} =\mathcal{O}(1)\Delta x^{2k}.&
\end{split}
\end{align*}
To conclude, we need to pass form $\p e$ to $e$ in the above expression. This has already been done for the diffusion term in Section 3 in \cite{Hailiang/Jue}
and for the convection term in the proof of Lemma 2.4 in \cite{Cockburn/Shu}. For the nonlocal term, we see that by \eqref{aaa}
\begin{align*}
\begin{split}
|\p e|^2_{H^{\lambda/2}(\R)}=|e|^2_{H^{\lambda/2}(\R)} -\mathcal{O}(1)\Delta x^{2k+2-\lambda},
\end{split}
\end{align*}
and the conclusion follows.
\end{proof}

\subsection{LDG method}

\begin{theorem}\label{LDGlin}
\emph{(Convergence)} Let $u\in H^{k+1}_{\mathrm{par}}(Q_T)$, $k\geq0$, be a solution of \eqref{linear} and $\tilde{\mathbf{w}}=(\tilde{u},\tilde{q})'\in C^1([0,T];V^k\cap L^2
)$ be a solution of \eqref{Ds2}. With $e_u=u-\tilde{u}$ and $e_q=q-\tilde{q}$,
\begin{align*}
\begin{split}
\int_\R e_u^2(x,T)+\int_0^T\int_{\R}e_q^2 +\Theta_T[\mathbf{e}] +bc_\lambda\int_0^T|e_u|^2_{H^{\lambda/2}(\R)} =\mathcal{O}(1)\Delta x^{2k}.
\end{split}
\end{align*}
\end{theorem}

\begin{proof}
Let us choose a test function $\mathbf{v}=(v_{u},v_{q})'$, $v_u,v_q\in V^{k}\cap L^{2}(\R)$, and define
\begin{align*}
\begin{split}
M_l[\mathbf{w},\mathbf{v}]&=\int_{0}^{T}\int_{\mathbb{R}}u_tv_u+\int_0^T\int_{\R}qv_q -\int_0^T\sum_{i\in\Z}\bigg(\hat{\mathbf{h}}(\mathbf{w}_i)'[\mathbf{v}_i]
+\int_{I_i}\mathbf{h}(\mathbf{w})'\partial_x\mathbf{v}\bigg).
\end{split}
\end{align*}
With this notation at hand, we can write \eqref{functionalB} as
\begin{align*}
\begin{split}
M_{LDG}[\mathbf{w},\mathbf{v}]=M_l[\mathbf{w},\mathbf{v}] +M_\Levy[\mathbf{w},\mathbf{v}],
\end{split}
\end{align*}
where $M_\Levy$ is defined in the previous proof. Proceeding as in the proof of Theorem \ref{DDGlin}, we find that
\begin{align}\label{mmm}
\begin{split}
M_{LDG}[\p\mathbf{e},\p\mathbf{e}]=M_{LDG}[\p\mathbf{e}-\mathbf{e},\p\mathbf{e}].
\end{split}
\end{align}
In  \cite{Cockburn/Shu} (Lemma 2.4) it is proved that
\begin{align}\label{k1}
M_l(\p \mathbf{e}-\mathbf{e},\p \mathbf{e})=\frac{1}{2}\Theta_T[\p \mathbf{e}] +\frac{1}{2}\int_0^T\int_\R\p e_q^2+\mathcal{O}(1)\Delta x^{2k}.
\end{align}
By \eqref{functionalB}, \eqref{e4}, and Lemma \ref{lemmanonloc},
\begin{align*}
\begin{split}
M_{LDG}[\p \mathbf{e},\p \mathbf{e}]=\int_{0}^{T}\int_{\mathbb{R}}(\p e^2_u)_t +\int_0^T\int_{\R}\p e_q^2+\Theta_T[\p
\mathbf{e}]+\frac{bc_\lambda}{2}\int_0^T|\p e_u|^2_{H^{\lambda/2}(\R)}.
\end{split}
\end{align*}
By inserting this inequality along with \eqref{k1} and \eqref{i3} into \eqref{mmm}, we find that
\begin{align*}
\begin{split}
\int_{0}^{T}\int_{\mathbb{R}}(\p e^2_u)_t&
+\frac{1}{2}\int_0^T\int_{\R}\p e_q^2+\frac{1}{2}\Theta_T[\p\mathbf{e}]\\
&+\frac{bc_\lambda}{4}\int_0^T|\p e_u|^2_{H^{\lambda/2}(\R)} +\frac{bc_\lambda}{4}\int_0^T|e_u|^2_{H^{\lambda/2}(\R)}=\mathcal{O}(1)\Delta x^{2k}.
\end{split}
\end{align*}
The conclusion now follows as in the proof of Theorem \ref{DDGlin}.
\end{proof}

\section{Convergence for nonlinear equations}
In the nonlinear case we will show that the DDG method converges towards an appropriately defined entropy solution of \eqref{1} whenever piecewise constant
elements are used. In what follows we need the functions
\begin{align*}
\eta_k(s)&=|s-k|,\\
\eta'_k(s)&=\text{sgn}(s-k),\\
q_k(s)&=\eta'_k(s)(f(s)-f(k)),\\
r_k(s)&=\eta'_k(s)(A(s)-A(k)).
\end{align*}
Remember that $A(u)=\int^ua$, and let $C^{1,\frac{1}{2}}(Q_T)$ denote the H\"{o}lder space of bounded functions $\phi:Q_T\rightarrow\R$ for which there is a
constant $c_\phi>0$ such that
\begin{align*}
\begin{split}
|\phi(x,t)-\phi(y,\tau)|\leq c_\phi\left[|x-y|+\sqrt{|t-\tau|}\right]\quad \text{for all}\quad (x,t),(y,\tau)\in Q_T.
\end{split}
\end{align*}
We now introduce the entropy formulation for \eqref{1}.
\begin{definition}\label{def:entropy}
A function $u\in L^{\infty}(Q_T)$ is a {\em $BV$ entropy solution} of the initial value problem \eqref{1} provided that the following conditions hold:
\begin{Definitions}
\item $u\in L^1(Q_T)\cap BV(Q_T)$;
\label{def:bv}
\item $A(u)\in C^{1,\frac{1}{2}}(Q_T)$;
\label{def:spacetimereg}
\item for all non-negative test functions $\varphi\in C_c^{\infty}(\R\times[0,T))$ and all $k\in\R$,
\begin{align*}
\begin{split}
\int_{Q_T}\eta_k(u)\varphi_t+q_k(u)\varphi_{x}
+r_{k}(u)\varphi_{xx}+\eta'_k(u)\Levy[u]\varphi\ \dx\dt&\\
+\int_{\R}\eta_k(u_0(x))\varphi(0,x)\dx&\geq 0.
\end{split}
\end{align*}
\label{def:entr}
\end{Definitions}
\end{definition}
This definition is a straightforward combination of the one of Wu and Yin \cite{Wu/Yin} (cf.~also \cite{Evje/Karlsen}) for degenerate convection-diffusion
equations ($b=0$) and the one of Cifani \emph{et al.}~\cite{Cifani/ERJ/KHK} for fractional conservation laws ($a\equiv0$). By the regularity of $\varphi$ and
$u$ and Lemma \ref{lemmanonloc}, each term in the entropy inequality \ref{def:entr} is well defined.

\begin{remark}
The $L^1$-contraction property (uniqueness) for $BV$ entropy solutions follows along the lines of \cite{KHK/Ulusoy}, since the $BV$-regularity of $u$ and the
$L^\infty$-bound on $A(u)_x$ makes it possible to recover from \ref{def:entr} the more precise entropy inequality utilized in \cite{KHK/Ulusoy} for $L^1\cap
L^\infty$ entropy solutions.
\end{remark}

We will now prove, under some additional assumptions, that the explicit DDG method with piecewise constant elements (i.e., $k=0$) converges to the $BV$ entropy
solution of \eqref{1}. In addition to convergence for the numerical method, this also gives the first existence result for entropy solutions of \eqref{1}.

\subsection{The explicit DDG method with piecewise constant elements}
When piecewise constant elements are used ($k=0$ in \eqref{udef}), equation \eqref{Ds2} takes the form
\begin{align*}
\int_{I_{i}}\hat{u}_{t}+\hat{f}(\hat{u}_{i+1})-\hat{f}(\hat{u}_{i})-\hat{h}(\hat{u}_{i+1}) +\hat{h}(\hat{u}_{i})=b\int_{I_{i}}\Levy[\hat{u}].
\end{align*}
Since $\hat{u}(x,t)=\sum_{i\in\mathbb{Z}} U_{i}(t)\mathbf{1}_{i}(x)$ (i.e., $\varphi_{0,i}=\mathbf{1}_i$, the indicator function of the interval $I_i$), we can
and will use the admissible flux $\hat{h}(u_i)=\frac{1}{\Delta x}[A(u_i)]$ (which satisfies \eqref{admiss} with $k=0$ and $\beta_0=1$) to rewrite the above
equation as
\begin{align*}
\Delta x\frac{d}{\dt}U_i+\hat{f}(U_i,U_{i+1})-\hat{f}(U_{i-1},U_{i}) -\frac{[A(U_{i+1})]}{\Delta x}+\frac{[A(U_{i})]}{\Delta
x}=b\sum_{j\in\Z}U_j\int_{I_{i}}\Levy[\mathbf{1}_{I_j}].
\end{align*}
For $\Delta t>0$ we set $t_n=n\Delta t$ for $n=\{0,\ldots,N\}$, $T=t_N$, and $\phi^n_i=\phi(x_i,t_n)$ for any function $\phi$. By a forward difference
approximation in time, we obtain the explicit numerical method
\begin{align}\label{Ds3}
\begin{split}
\frac{U_i^{n+1}-U_i^{n}}{\Delta t}&+\frac{\hat{f}(U_i^n,U_{i+1}^n) -\hat{f}(U_{i-1}^n,U_{i}^n)}{\Delta x}\\&-\frac{A(U^n_{i+1})-A(U^n_{i})}{\Delta
x^2}+\frac{A(U^n_{i})-A(U^n_{i-1})}{\Delta x^2}=\frac{b}{\Delta x}\sum_{j\in\Z}G^i_jU_j^n,
\end{split}
\end{align}
where the weights $G^i_j=\int_{I_{i}}\Levy[\mathbf{1}_{I_j}]\text{ for
  all $(i,j)\in\Z\times\Z$.}$ All relevant properties of these
weights are collected in Lemma \ref{matG}. Next we define
\begin{gather*}
D_\pm U_i=\pm\frac{1}{\Delta x}\left(U_{i\pm 1}-U_{i}\right)\quad\text{and}\quad \Levy\langle U^{n}\rangle_{i}=\frac{1}{\Delta
x}\int_{I_{i}}\Levy[\bar{U}^{n}]\dx =\frac{1}{\Delta x}\sum_{j\in\mathbb{Z}}G^i_{j}U_{j}^n,
\end{gather*}
where $\bar{U}^n$ is the piecewise constant interpolant of $U^n$:
$$
\bar{U}^n(x)=U^n_i, \qquad x\in[x_i,x_{i+1}).
$$
The explicit numerical method we study can then be written as
\begin{align}\label{scheme}
\begin{cases}
\frac{U_i^{n+1}-U_i^{n}}{\Delta t}+D_{-} \Big[\hat{f}(U_i^n,U_{i+1}^n)-D_{+}A(U^n_i)\Big]
=b\Levy\langle U^{n}\rangle_{i},\\[0.2cm]
U^0_i=\frac{1}{\Delta x}\int_{I_i}u_0(x)\dx.
\end{cases}
\end{align}
As we will see in what follows, the low-order difference method \eqref{scheme} allows for a complete convergence analysis for general nonlinear equations of
the form \eqref{1}.

Let us now prove that the difference scheme \eqref{scheme} is conservative (\textbf{P}.1), monotone (\textbf{P}.2), and translation invariant (\textbf{P}.3).
\begin{itemize}
\item[(\textbf{P}.1)] Assume $\bar U^n\in L^1(\R)\cap BV(\R)$. By
  Lemma \ref{lemmanonloc}
\begin{align}\label{crucial}
\sum_{i\in\mathbb{Z}}\sum_{j\in\mathbb{Z}}|G^i_jU_{j}^n|\leq \int_\R|\Levy[\bar{U}^n(x)]|\dx\leq c_\lambda
C\|\bar{U}^n\|_{L^1(\R)}^{1-\lambda}|\bar{U}^n|_{BV(\R)}^{\lambda},
\end{align}
and hence we can revert the order of summation to obtain
\begin{align*}
\sum_{i\in\mathbb{Z}}\sum_{j\in\mathbb{Z}}G^i_{j}U_{j}^n= \sum_{j\in\mathbb{Z}}U_{j}^n\sum_{i\in\mathbb{Z}}G^i_{j}=0
\end{align*}
since $\sum_{i\in\mathbb{Z}}G^i_{j}=0$ by Lemma \eqref{matG}. By summing over all $i\in\mathbb{Z}$ on each side of \eqref{scheme}, we then find that
\begin{align*}
\begin{split}
\sum_{i\in\Z}U_i^{n+1}=\sum_{i\in\Z}\Big(U_i^{n} +\frac{\Delta t}{\Delta x}\sum_{j\in\Z}G^i_jU_j^n\Big)=\sum_{i\in\Z}U_i^{n}.
\end{split}
\end{align*}
\item[(\textbf{P}.2)]We show that $U^{n+1}_i$ is an increasing function
of all $\{U^{n}_i\}_{i\in\Z}$. First note that
\begin{align*}
\begin{split}
\frac{\partial U^{n+1}_i}{\partial U^{n}_j}\geq0 \quad\text{for} \quad i\neq j,
\end{split}
\end{align*}
since $\hat{f}$ is monotone and $G^i_j\geq0$ for $i\neq j$ by Lemma \ref{matG}. By  Lemma \ref{matG}  we also  see that $G_i^i=-d_\lambda\Delta
x^{1-\lambda}\leq0,$ and hence
\begin{align*}
\begin{split}
\frac{\partial U^{n+1}_i}{\partial U^{n}_i}=&\ 1-\frac{\Delta t}{\Delta
x}\Big[\partial_{u_1}\hat{f}(U_i^n,U_{i+1}^n)-\partial_{u_2}\hat{f}(U_{i-1}^n,U_{i}^n)\Big]\\
&-2\frac{\Delta t}{\Delta x^2}a(U^n_{i})-\frac{\Delta t}{\Delta x^{\lambda}}d_\lambda.
\end{split}
\end{align*}
Here $\partial_{u_i}\hat{f}$ denotes the derivative of $\hat{f}(u_1,u_2)$ w.r.t.~$u_i$ for $i=1,2$. Therefore the following CFL condition makes the explicit
method \eqref{scheme} monotone:
\begin{align}\label{cfl}
\begin{split}
\frac{\Delta t}{\Delta x}\bigg(\|\partial_{u_1}\hat{f}\|_{L^\infty(\R)} +\|\partial_{u_2}\hat{f}\|_{L^\infty(\R)}\bigg) +\frac{2\Delta t}{\Delta
x^2}\|a\|_{L^\infty(\R)} +d_{\lambda}\frac{\Delta t}{\Delta x^\lambda}\leq1.
\end{split}
\end{align}

\item[(\textbf{P}.3)] Translation invariance ($V^0_i=U^0_{i+1}$ implies
$V_i^n=U^n_{i+1}$) is straightforward since \eqref{scheme} does not depend explicitly on a grid point $x_i$.
\end{itemize}

\begin{remark}
For several well known numerical fluxes $\hat f$ (i.e.~Godunov, Engquist-Osher, Lax-Friedrichs, \emph{etc.}), we may replace
$$\|\partial_{u_1}\hat{f}\|_{L^\infty(\R)}
+\|\partial_{u_2}\hat{f}\|_{L^\infty(\R)}$$ in the above CFL condition by the Lipschitz constant of the original flux $f$.
\end{remark}

In the following, we always assume that the CFL condition \eqref{cfl} holds.

\subsection{Further properties of the explicit DDG method \eqref{scheme}}
Define
\begin{align*}
\begin{split}
\|U\|_{L^{1}(\mathbb{Z})}=\sum_{i\in\Z}|U_i|,\quad
\|U\|_{L^{\infty}(\mathbb{Z})}=\sup_{i\in\Z}|U_i|,\quad\text{and}\quad|U|_{BV(\mathbb{Z})}=\sum_{i\in\Z}|U_{i+1}-U_{i}|.
\end{split}
\end{align*}
\begin{lemma}\label{stab}\*
\begin{itemize}
\item[\emph{i)}]$\|U^n\|_{L^{1}(\Z)}\leq\|u_0\|_{L^{1}(\R)}$,
\item[\emph{ii)}]$\|U^n\|_{L^{\infty}(\Z)}\leq\|u_0\|_{L^{\infty}(\R)}$,
\item[\emph{iii)}]$|U^n|_{BV(\Z)}\leq|u_0|_{BV(\R)}$.
\end{itemize}
\end{lemma}
\begin{proof}
Since the numerical method \eqref{scheme} is conservative monotone and translation invariant, the results due to Crandall-Tartar
\cite{Crandall/Tartar,Evje/Karlsen} and Lucier \cite{Lucier,Evje/Karlsen} apply.
\end{proof}

For all $(x,t)\in R^n_i=[x_i,x_{i+1})\times[t_n,t_{n+1})$, let $\hat{u}_{\Delta x}(x,t)$ be the time-space bilinear interpolation of $U^n_i$, i.e.
\begin{align}\label{interp}
\begin{split}
\hat{u}_{\Delta x}(x,t)=U_i^n&+(U_{i+1}^n-U_i^n)\left(\frac{x-i\Delta x}{\Delta x}\right)+(U_i^{n+1}-U_i^n)\left(\frac{t- n\Delta t}{\Delta t}\right)\\
&+(U_{i+1}^{n+1}-U_i^{n+1}-U_{i+1}^n+U_i^n)\left(\frac{x-i\Delta x}{\Delta x}\right)\left(\frac{t-n\Delta t}{\Delta t}\right).
\end{split}
\end{align}
Note that $\hat{u}_{\Delta x}$ is continuous and a.e.~differentiable on $Q_T$. We need the above bilinear interpolation -- rather than a piecewise constant one
-- to prove the H\"{o}lder regularity in \ref{def:spacetimereg}. We will show that the functions $A(\hat{u}_{\Delta x})$ enjoy H\"{o}lder regularity as in
\ref{def:spacetimereg}, and then via an Ascoli-Arzel\`{a} type of argument, so does the limit $A(u)$.

The following lemmas which are needed in the proof of Theorem \ref{th:existence}, are nonlocal generalizations of the ones proved in \cite{Evje/Karlsen}. In
what follows we assume $f\in C^1(\R)$, and note that the general case follows by approximation as in \cite{Evje/Karlsen}.

\begin{lemma}\label{lemmaAP1}
\begin{align}\label{first}
\begin{split}
\bigg\|\hat{f}(U_i^n,U_{i+1}^n)~-~&D_{+}A(U^n_i)-\sum_{k=-\infty}^i\sum_{j\in\Z}G^k_jU_j^n\bigg\|_{L^{\infty}(\Z)}\\
&\leq\bigg\|\hat{f}(U_i^0,U_{i+1}^0)-D_{+}A(U^0_i)-\sum_{k=-\infty}^i\sum_{j\in\Z}G^k_jU_j^0\bigg\|_{L^{\infty}(\Z)},
\end{split}
\end{align}
\begin{align}\label{second}
\begin{split}
\bigg|\hat{f}(U_i^n,U_{i+1}^n)~-~&D_{+}A(U^n_i)-\sum_{k=-\infty}^i\sum_{j\in\Z}G^k_jU_j^n\bigg|_{BV(\Z)}\\
&\leq\bigg|\hat{f}(U_i^0,U_{i+1}^0)-D_{+}A(U^0_i)-\sum_{k=-\infty}^i\sum_{j\in\Z}G^k_jU_j^0\bigg|_{BV(\Z)}.
\end{split}
\end{align}
\end{lemma}
\begin{proof}
\emph{Inequality \eqref{first}.} Let us start by defining $V_i^n=\frac{\Delta x}{\Delta t}\sum_{k=-\infty}^i(U_k^n-U_k^{n-1})$. This sum is finite since
$U^n\in L^1(\Z)$ for all $n\geq0$. If we use \eqref{scheme}, we can write
\begin{align}\label{a10}
\begin{split}
V_i^{n+1}=-\Big[\hat{f}(U_i^n,U_{i+1}^n)-D_{+}A(U^n_i)\Big] +\sum_{k=-\infty}^i\sum_{j\in\Z}G^k_jU_j^n.
\end{split}
\end{align}
Here we have used that $U^n\in L^1(\Z)\cap BV(\Z)$, $f$ and $A$ are Lipschitz continuous, and $f(0)=0$ to conclude that the sum
$\sum_{k=-\infty}^iD_{-}[\hat{f}(U_j^n,U_{j+1}^n) -D_{+}A(U^n_j)]$ is finite and has value $[\hat{f}(U_i^n,U_{i+1}^n)-D_{+}A(U^n_i)]$. Next we rewrite the
right-hand side of \eqref{a10} in terms of $\{V_i^n\}_{i\in\Z}$. By \eqref{a10},
\begin{align}\label{a1}
\begin{split}
V_i^{n+1}&=V^n_i-\Big[\hat{f}(U^n_i,U^n_{i+1})-\hat{f}(U^{n-1}_i,U^{n-1}_{i+1})
-D_+(A(U^n_i)-A(U^{n-1}_i))\Big]\\
&\qquad+\sum_{k=-\infty}^{i}\sum_{j\in\Z}G^k_j(U_j^{n}-U_j^{n-1}).
\end{split}
\end{align}
We prove that
\begin{align}\label{rights}
\begin{split}
\sum_{k=-\infty}^{i}\sum_{j\in\Z}G^k_j(U_j^{n}-U_j^{n-1})= \frac{\Delta t}{\Delta x}\sum_{j\in\Z}G^i_jV_j^{n}.
\end{split}
\end{align}
Indeed, note that $D_-V^n_j=\frac{1}{\Delta t}\left(U_j^{n}-U_j^{n-1}\right)$ and
\begin{align*}
\begin{split}
\sum_{j\in\Z}G^k_jV_{j-1}^{n}=\sum_{j\in\Z}G^{k}_{j+1}V_{j}^{n}=\sum_{j\in\Z}G^{k-1}_jV_{j}^{n}
\end{split}
\end{align*}
since $G^{k}_{j+1}=G^{k-1}_j$. Thus,
\begin{align*}
\begin{split}
\sum_{k=-\infty}^{i}\sum_{j\in\Z}G^k_j(U_j^{n}-U_j^{n-1})&=
\Delta t\sum_{k=-\infty}^{i}\sum_{j\in\Z}G^k_jD_-V^n_j\\
&=\frac{\Delta t}{\Delta x}\sum_{k=-\infty}^{i}\sum_{j\in\Z}G^k_j(V_j^{n}-V_{j-1}^{n})\\
&=\frac{\Delta t}{\Delta x}\sum_{k=-\infty}^{i}
\sum_{j\in\Z}G^k_jV_j^{n}-\frac{\Delta t}{\Delta x}\sum_{k=-\infty}^{i}\sum_{j\in\Z}G^k_jV_{j-1}^{n}\\
&=\frac{\Delta t}{\Delta x}\sum_{k=-\infty}^{i}
\sum_{j\in\Z}G^k_jV_j^{n}-\frac{\Delta t}{\Delta x}\sum_{k=-\infty}^{i}\sum_{j\in\Z}G^{k-1}_jV_{j}^{n}\\
&=\frac{\Delta t}{\Delta x}\sum_{k=-\infty}^{i}
\sum_{j\in\Z}G^k_jV_j^{n}-\frac{\Delta t}{\Delta x}\sum_{k=-\infty}^{i-1}\sum_{j\in\Z}G^{k}_jV_{j}^{n}\\
&=\frac{\Delta t}{\Delta x}\sum_{j\in\Z}G^i_jV_j^{n}.
\end{split}
\end{align*}
Using Taylor expansions, we can replace the nonlinearities $\hat{f}, A$ with linear approximations as follows. We write
\begin{align}\label{r1}
\begin{split}
\hat{f}(U^n_i,U^n_{i+1})-\hat{f}(U^{n-1}_i,U^{n-1}_{i+1}) =\Delta t\hat{f}_{1,i}^nD_-V^n_i+\Delta t\hat{f}_{2,i}^nD_-V^n_{i+1},
\end{split}
\end{align}
where $\hat{f}_{1,i}^n=\partial_{1}\hat{f}(\alpha_i^n,U^n_{i+1})$, $\hat{f}_{2,i}^n =\partial_{2}\hat{f}(U^{n-1}_{i},\tilde{\alpha}_{i+1}^n)$ and
$\alpha_i^n,\tilde{\alpha}_i^n\in(U^{n-1}_i,U^n_i)$. Similarly, we write
\begin{align}\label{r2}
\begin{split}
A(U^n_i)-A(U^{n-1}_i)=a(\beta_i^n)(U^n_i-U^{n-1}_i)=\Delta ta_i^nD_-V^n_i,
\end{split}
\end{align}
where $a_i^n=a(\beta_i^n)$ and $\beta_i^n\in(U^{n-1}_i,U^n_i)$. Inserting \eqref{rights} and \eqref{r1}-\eqref{r2} into expression \eqref{a1} returns
\begin{align}\label{888}
\begin{split}
V_i^{n+1}&=V^n_i-\Delta t(\hat{f}_{1,i}^nD_-V^n_i +\hat{f}_{2,i}^nD_-V^n_{i+1})+\Delta tD_+(a_i^nD_-V_i^n) +\frac{\Delta t}{\Delta x}\sum_{j\in\Z}G^i_jV_j^{n}
\end{split}
\end{align}
or
\begin{align}\label{a2}
\begin{split}
V^{n+1}_i=A^n_iV^n_{i-1}+B^n_iV^n_i+C^n_iV^n_{i+1} +\frac{\Delta t}{\Delta x}\sum_{j\in\Z}G^i_jV_j^{n},
\end{split}
\end{align}
where
\begin{align*}
\begin{split}
A^n_i&=\left[\frac{\Delta t}{\Delta x}\hat{f}_{1,i}^n+\frac{\Delta t}{\Delta x^2}a_i^n\right],\\
B^n_i&=\left[1-\frac{\Delta t}{\Delta x}(\hat{f}_{1,i}^n-\hat{f}_{2,i}^n)
-\frac{\Delta t}{\Delta x^2}(a_i^n+a^n_{i+1})\right],\\
C^n_i&=\left[\frac{\Delta t}{\Delta x^2}a_{i+1}^n-\frac{\Delta t}{\Delta x}\hat{f}_{2,i}^n\right].
\end{split}
\end{align*}
Since $\hat{f}$ is monotone and $a\geq0$, $A^n_i,C^n_i\geq0$. Moreover, $B^n_i+\frac{\Delta t}{\Delta x}G^i_i\geq0$ since the CFL condition \eqref{cfl} holds
true. Thus, since \eqref{a2} is conservative, monotone, and translation invariant (cf.~the proof of Lemma \ref{stab}),
$\|V^{n}\|_{L^\infty(\Z)}\leq\ldots\leq\|V^{1}\|_{L^\infty(\Z)}$, and the conclusion follows from \eqref{a10}.

\emph{Inequality \eqref{second}.} Let us introduce $Z^n_i=V^n_i-V^n_{i-1}$. Note that, since $G^{i-1}_{j-1}=G_j^i$ for all $(i,j)\in\Z\times\Z$,
\begin{align*}
\begin{split}
\sum_{j\in\Z}\left(G^i_jV_j^{n}-G^{i-1}_jV_j^{n}\right)= \sum_{j\in\Z}\left(G^i_jV_j^{n}-G^{i}_jV_{j-1}^{n}\right)=\sum_{j\in\Z}G^i_jZ_j^{n}.
\end{split}
\end{align*}
Thus, \eqref{888} can be rewritten as
\begin{align*}
\begin{split}
Z_i^{n+1}&=Z^n_i-\Delta tD_-(\hat{f}_{1,i}^nZ^n_i +\hat{f}_{2,i}^nZ^n_{i+1})+\Delta t D_-D_+(a_i^nZ_i^n) +\frac{\Delta t}{\Delta x}\sum_{j\in\Z}G^i_jZ_j^{n}
\end{split}
\end{align*}
or
\begin{align}\label{a11}
\begin{split}
Z^{n+1}_i=\bar{A}^n_iZ^n_{i-1}+\bar{B}^n_iZ^n_i+\bar{C}^n_iZ^n_{i+1} +\frac{\Delta t}{\Delta x}\sum_{j\in\Z}G^{i}_jZ_j^{n},
\end{split}
\end{align}
where $\bar{A}^n_i,\bar{B}^n_i,\bar{C}^n_i$ have similar properties as $A^n_i,B^n_i,C^n_i$. Proceeding as in the first part of the proof, \eqref{a11} can be
shown to be conservative, monotone, and translation invariant. Thus $\|Z^{n}\|_{L^1(\Z)}\leq\ldots\leq\|Z^{1}\|_{L^1(\Z)}$, and the conclusion follows from
\eqref{a10}.  We refer to \cite{Evje/Karlsen} for the precise details concerning $\bar{A}^n_i,\bar{B}^n_i,\bar{C}^n_i$.
\end{proof}

The next lemma ensures that the numerical solutions are uniformly $L^1$-Lipschitz in time (and hence $BV$ in both space and time by Lemma \ref{stab}).

\begin{lemma}\label{Liptime}
\begin{align*}
\begin{split}
\sum_{i\in\Z}|U_i^m-U_i^n|\leq\Big|\hat{f}(U_i^0,U_{i+1}^0)-D_{+}A(U^0_i) -\sum_{k=-\infty}^{i}\sum_{j\in\Z}G^k_jU^0_j\Big|_{BV(\R)}\frac{\Delta t}{\Delta
x}|m-n|.
\end{split}
\end{align*}
\end{lemma}
\begin{proof}
Let us assume that $m>n$, the case $m<n$ is analogous. Note that
\begin{align*}
\begin{split}
&\sum_{i\in\Z}|U_i^m-U_i^n|\leq\sum_{l=n}^{m-1}\sum_{i\in\Z}|U_i^{l+1}-U_i^{l}|\\
&\leq\Delta t\sum_{l=n}^{m-1}\sum_{i\in\Z}\Bigg|D_{-}\Big[\hat{f}(U_i^l,U_{i+1}^l) -D_{+}A(U^l_i)\Big]-\frac{1}{\Delta x}\sum_{j\in\Z}G^i_jU_j^l\Bigg|.
\end{split}
\end{align*}
Since $D_-\left(\sum_{k=-\infty}^{i}\sum_{j\in\Z}G^k_jU_j\right) =\frac{1}{\Delta x}\sum_{j\in\Z}G^i_jU_j$,
\begin{align*}
\begin{split}
&\sum_{i\in\Z}|U_i^m-U_i^n|\\
&\leq\Delta t\sum_{l=n}^{m-1}\sum_{i\in\Z} \Bigg|D_{-}\Big[\hat{f}(U_i^l,U_{i+1}^l)-D_{+}A(U^l_i) -\sum_{k=-\infty}^{i}\sum_{j\in\Z}G^k_jU^l_j\Big]\Bigg|.
\end{split}
\end{align*}
To conclude, use \eqref{second}.
\end{proof}

We now show that the numerical solutions satisfy a discrete version of \ref{def:spacetimereg}.

\begin{lemma}\label{lemmalip}
If $|\hat{f}(U_i^0,U_{i+1}^0)-D_{+}A(U^0_i) -\sum_{k=-\infty}^i\sum_{j\in\Z}G^k_jU_j^0|_{BV(\Z)}<\infty$, then
\begin{align*}
\begin{split}
|A(U^m_i)-A(U^n_j)|=\mathcal{O}(1) \left[|i-j|\Delta x+\sqrt{|m-n|\Delta t}\right].
\end{split}
\end{align*}
\end{lemma}
\begin{proof}
Let us write
\begin{align*}
\begin{split}
\left|A(U^m_i)-A(U^n_j)\right|\leq\left|A(U^m_i)-A(U^m_j)\right| +\left|A(U^m_j)-A(U^n_j)\right|=I_1+I_2.
\end{split}
\end{align*}
We first estimate the term $I_1$, then the term $I_2$.

\emph{Estimate of $I_1$.} Using \eqref{first}, \eqref{crucial}, Lemma \ref{stab} ii), and the fact that $f$ is Lipschitz continuous,
\begin{align*}
\begin{split}
\left\|D_{+}A(U^m_i)\right\|_{L^{\infty}(\Z)}\leq\ & \bigg\|\hat{f}(U_i^0,U_{i+1}^0)-D_{+}A(U^0_i)
-\sum_{k=-\infty}^i\sum_{j\in\Z}G^k_jU_j^0\bigg\|_{L^{\infty}(\Z)}\\
&+\bigg\|\hat{f}(U_i^m,U_{i+1}^m)\bigg\|_{L^{\infty}(\Z)} +\bigg\|\sum_{k=-\infty}^i\sum_{j\in\Z} G^k_jU_j^m\bigg\|_{L^{\infty}(\Z)}=\mathcal{O}(1).
\end{split}
\end{align*}
Hence $I_1=\mathcal{O}(1)|i-j|\Delta x$.

\emph{Estimate of $I_2$.} Take a test function $\phi\in C^1_c(\R)$, and let $\phi_i=\phi(i\Delta x)$. Let us assume $m>n$ (the case $m<n$ is analogous). Using
\eqref{888} we find that
\begin{align*}
\begin{split}
&\left|\Delta x\sum_{i\in\Z}\phi_i\left(V_i^m-V_i^n\right)\right|
\\ & =\Delta x\left|\sum_{l=n}^{m-1}\sum_{i\in\Z}\phi_i\left(V^{l+1}_i-V^l_i\right)\right|\\
&=\Delta x\sum_{l=n}^{m-1}\sum_{i\in\Z}\phi_i
\Big|(\hat{f}_{1,i}^nD_-V^l_i+\hat{f}_{2,i}^nD_-V^l_{i+1})+D_+(a_i^nD_-V_i^l)\Big|\\
&\qquad\qquad +\Delta t\left|\sum_{l=n}^{m-1} \sum_{i\in\Z}\phi_i\sum_{j\in\Z}G^i_jV_j^{l}\right|=C_1+C_2.
\end{split}
\end{align*}
We use summation by parts to move $D_+$ onto $\phi_i$ and the fact that $\hat{f}_{1,i}^n$ $\hat{f}_{2,i}^n$ $a_i^n$ and $|V^l|_{BV(\Z)}$ are uniformly bounded
to arrive that
\begin{align*}
\begin{split}
C_1=\mathcal{O}(1)\Delta t(m-n)\left(\|\phi\|_{L^\infty(\R)} +\|\phi'\|_{L^\infty(\R)}\right).
\end{split}
\end{align*}
For more details, see \cite{Evje/Karlsen}. Then by \eqref{crucial}, $\sum_{i\in\Z}\sum_{j\in\Z}|G^i_jV_j^{l}|=\mathcal{O}(1)$, and hence
\begin{align}\label{kkk}
\begin{split}
C_2\leq\Delta t\|\phi\|_{L^\infty(\R)}\sum_{l=n}^{m-1} \sum_{i\in\Z}\sum_{j\in\Z}|G^i_jV_j^{l}| =\mathcal{O}(1)\Delta t(m-n)\|\phi\|_{L^\infty(\R)}.
\end{split}
\end{align}
Therefore,
\begin{align}\label{n1}
\begin{split}
\left|\Delta x\sum_{i\in\Z}\phi_i\left(V_i^m-V_i^n\right)\right| =\mathcal{O}(1)\Delta t(m-n)\Big[\|\phi\|_{L^\infty(\R)}+\|\phi'\|_{L^\infty(\R)}\Big].
\end{split}
\end{align}
The above inequality is exactly expression (40) in \cite{Evje/Karlsen}. From now on the proof continues as in \cite{Evje/Karlsen}. Loosely speaking we take an
appropriate sequence of test functions $\phi_\eps\in C^1_c(\R)$ to deduce from \eqref{n1} that
\begin{align*}
\begin{split}
\Delta x\sum_{i\in\Z}|V_i^{m}-V_i^{n}|=\mathcal{O}(1)\sqrt{(m-n)\Delta t}.
\end{split}
\end{align*}
By \eqref{a10}, Lemma \ref{Liptime}, and inequality \eqref{kkk} we also find that
 $$\Delta
x\sum_{i\in\Z}|V_i^{m}-V_i^{n}|=\mathcal{O}(1)(m-n)\Delta t+\Delta x\sum_{i\in\Z}|D_+A(U_i^{m})-D_+A(U_i^{n})|,$$ and hence $\Delta
x\sum_{i\in\Z}|D_+A(U_j^{m})-D_+A(U_j^{n})|=\mathcal{O}(1)\sqrt{(m-n)\Delta
  t}$. We conclude by noting that
\begin{align*}
\begin{split}
I_2=|A(U^m_j)-A(U^n_j)|&=\Delta x\left|\sum_{i=-\infty}^jD_+A(U_i^m)
-\sum_{i=-\infty}^jD_+A(U_i^n)\right|\\
&\leq\Delta x\sum_{i\in\Z}|D_+A(U_i^{m})-D_+A(U_i^{n})| =\mathcal{O}(1)\sqrt{(m-n)\Delta t}.
\end{split}
\end{align*}
\end{proof}

Next we show that the numerical method \eqref{scheme} satisfies a cell entropy inequality, which is  a discrete version of \ref{def:entr}.

\begin{lemma}
Let $k\in\R$ and $\eta_i^n=|U^n_i-k|$. Then
\begin{align}\label{entropcell}
\begin{split}
\eta_i^{n+1}-\eta_i^{n}+\Delta tD_{-}Q_i^n-\Delta tD_{-}D_{+}|A(U^n_i)-A(k)| &\leq\Delta t\eta'_k(U^{n+1}_i)\Levy\langle U^n\rangle_i,
\end{split}
\end{align}
where $Q_i^n=\hat{f}(U_i^n\vee k,U_{i+1}^n\vee k) -\hat{f}(U_i^n\wedge k,U_{i+1}^n\wedge k)$.
\end{lemma}

\begin{proof}
Let us introduce the notation $a\wedge b=\min\{a,b\}$ and $a\vee b=\max\{a,b\}$. Note that $\eta_i^n=(U_i^{n}\vee k)-(U_i^{n}\wedge k)$. Since the numerical
method \eqref{scheme} is monotone,
\begin{align*}
\begin{split}
&\frac{(U_i^{n+1}\vee k)-(U_i^{n}\vee k)}{\Delta t} +\frac{\hat{f}(U_i^n\vee k,U_{i+1}^n\vee k)-\hat{f}(U_{i-1}^n\vee k,U_{i}^n\vee k)}{\Delta
x}\\&-\frac{A(U^n_{i+1}\vee k)-A(U^n_{i}\vee k)}{\Delta x^2} +\frac{A(U^n_{i}\vee k)-A(U^n_{i-1}\vee k)}{\Delta x^2}\\
&\qquad\qquad\qquad\qquad\qquad\qquad\qquad\qquad\qquad\leq\Delta t\mathbf{1}_{(k,+\infty)}(U_{i}^{n+1})\Levy\langle U^{n}\rangle_{i}
\end{split}
\end{align*}
and
\begin{align*}
\begin{split} &\frac{(U_i^{n+1}\wedge k)-(U_i^{n}\wedge k)}{\Delta t} +\frac{\hat{f}(U_i^n\wedge k,U_{i+1}^n\wedge k) -\hat{f}(U_{i-1}^n\wedge
k,U_{i}^n\wedge k)}{\Delta
x}\\&-\frac{A(U^n_{i+1}\wedge k)-A(U^n_{i}\wedge k)}{\Delta x^2} +\frac{A(U^n_{i}\wedge k)-A(U^n_{i-1}\wedge k)}{\Delta x^2}\\
&\qquad\qquad\qquad\qquad\qquad\qquad\qquad\qquad\qquad\geq\Delta t\mathbf{1}_{(-\infty,k)}(U_{i}^{n+1})\Levy\langle U^{n}\rangle_{i}.
\end{split}
\end{align*}
To conclude, subtract the above inequalities.
\end{proof}

\subsection{Convergence of the DDG method}
We are now in position to prove convergence of the fully explicit numerical method \eqref{scheme} to a $BV$ entropy solution of \eqref{1}. Let us introduce
$\mathcal{B}$ (cf.~\cite{Evje/Karlsen}), the space of all functions $z:\R\rightarrow\R$ such that
\begin{align*}
\begin{split}
\left|f(z)-\partial_xA(z)-\int^x\Levy[z]\right|_{BV(\R)}<\infty.
\end{split}
\end{align*}
In the following theorem we choose the initial datum to be in $L^1(\R)\cap BV(\R)\cap \mathcal{B}$, which is done to make sense to the right-hand side of
\eqref{second}. Note that whenever $z\in L^1(\R)\cap BV(\R)$, $\Levy[z]\in L^1(\R)$ by Lemma \ref{lemmanonloc}, and hence
$$
\left|\int^x\Levy[z]\right|_{BV(\R)} =\left\|\frac{d}{\dx}\int^x\Levy[z]\right\|_{L^1(\R)} =\|\Levy[z]\|_{L^1(\R)}<\infty.$$

\begin{theorem}[Convergence for DDG]\label{th:existence}
Suppose $u_0\in L^1(\R)\cap BV(\R)\cap \mathcal{B}$, and let $\hat{u}_{\Delta x}$ be the interpolant \eqref{interp} of the solution of the explicit DGG scheme
\eqref{scheme}. Then there is a subsequence of $\{\hat{u}_{\Delta x}\}$ and a function $u\in L^1(Q_T)\cap BV(Q_T)$ such that (a) $\hat{u}_{\Delta x}\to u$ in
$L_{loc}^1(Q_T)$ as $\Delta x\ra0$; (b) $u$ is a $BV$ entropy solution of \eqref{1}.
\end{theorem}

\begin{corollary}[Existence]
If $u_0\in L^1(\R)\cap BV(\R)\cap \mathcal{B}$, then there exists a BV entropy solution of \eqref{1}.
\end{corollary}

\begin{proof}[Proof of Theorem \ref{th:existence}]
We will prove strong $L_{loc}^1$ compactness, and hence we need the following estimates uniformly in $\Delta x>0$:
\begin{itemize}
\item[\emph{i)}]$\|\hat{u}_{\Delta x}\|_{L^\infty(Q_T)}\leq C,$
\item[\emph{ii)}]$\|\hat{u}_{\Delta x}\|_{BV(Q_T)}\leq C.$
\end{itemize}
Estimate $\emph{i)}$ is a consequence of Lemma \ref{stab} and \eqref{interp}, while estimate $\emph{ii)}$ comes from the following computations
(cf.~\cite{Evje/Karlsen} for more details). Using the interpolation \eqref{interp}, we find that
\begin{align*}
\begin{split}
\int_{Q_T}|\hat{u}_x| &\leq\frac{\Delta t}{2}\sum_{n=0}^{N-1}\sum_{i\in\Z}\left|U^n_{i+1} -U^n_{i}\right|+\frac{\Delta t}{2}\sum_{n=0}^{N-1}
\sum_{i\in\Z}\left|U^{n+1}_{i+1}-U^{n+1}_{i}\right|\\
&\leq T|U^0|_{BV(\Z)}.
\end{split}
\end{align*}
Note that Lemma \ref{stab} has been used in the second inequality. Similarly,
\begin{align*}
\begin{split}
\int_{Q_T}|\hat{u}_t| &\leq\frac{\Delta x}{2}\sum_{n=0}^{N-1} \sum_{i\in\Z}\left|U^{n+1}_{i}-U^n_{i}\right|+\frac{\Delta x}{2}
\sum_{n=0}^{N-1}\sum_{i\in\Z}\left|U^{n+1}_{i+1}-U^{n}_{i+1}\right|\\
&\leq T\Big|\hat{f}(U_i^0,U_{i+1}^0)-D_{+}A(U^0_i) -\sum_{h=-\infty}^{i}\sum_{j\in\Z}G^h_jU^0_j\Big|_{BV(\Z)},
\end{split}
\end{align*}
where Lemma \ref{Liptime} has been used in the second inequality. Hence, there exists a sequence $\{\hat{u}_{\Delta x_i}\}_{i\in\N}$ which converges in
$L^1_{loc}(Q_T)$ to a limit
$$
u\in L^1(Q_T)\cap BV(Q_T).
$$

Next we check that the limit $u$ satisfies \ref{def:spacetimereg}. We define $w_{\Delta x}=A(\hat{u}_{\Delta x})$. Note that $A(\hat{u}_{\Delta x})\rightarrow
A(u)$ a.e. since $\hat{u}_{\Delta x}\rightarrow u$ a.e. (up to a subsequence) and $A$ is continuous.
 Now choose $(x,t),(y,\tau),(j,n),(i,m)$ such that
$(x,t)\in R_j^n$ and $(y,\tau)\in R_i^m$ for $R^n_i=[x_i,x_{i+1})\times[t_n,t_{n+1})$. Then,
\begin{align*}
\begin{split}
|w_{\Delta x}(y,\tau)-w_{\Delta x}(x,t)|&\leq|w_{\Delta x}(y,\tau)-w_{\Delta x}(i\Delta x,m\Delta t)|\\
&\ +|w_{\Delta x}(i\Delta x,m\Delta t)-w_{\Delta x}(j\Delta x,n\Delta t)|\\
&\ +|w_{\Delta x}(j\Delta x,n\Delta t)-w_{\Delta x}(x,t)|\\
&=I_1+I_2+I_3.
\end{split}
\end{align*}
Note that by Lemma \ref{lemmalip}, $I_2=\mathcal{O}(1)(|i-j|\Delta x+\sqrt{|m-n|\Delta t})$, while by Lemma \ref{lemmalip} again, \eqref{interp}, and $A'=a\in
L^\infty$, $I_1+I_3=\mathcal{O}(1)(\Delta x+\sqrt{\Delta t})$. Thus
\begin{align*}
|w_{\Delta x}(y,\tau)-w_{\Delta x}(x,t)|=\mathcal{O}(1)\left[|y-x|+\sqrt{|\tau-t|}+\Delta x+\sqrt{\Delta t}\right].
\end{align*}
We also have that $w_{\Delta x}=A(\hat{u}_{\Delta x})$ is uniformly bounded since $A$ is Lipschitz and  $\hat u_{\Delta x}$ is uniformly bounded. By
essentially repeating the proof of the Ascoli-Arzel\`{a} compactness theorem, we can now deduce the existence of a subsequence $\{w_{\Delta x}\}$ converging
locally uniformly towards the limit $A(u)$. By the estimates on $w_{\Delta x}$, it then follows that
\begin{align}\label{kkjj}
A(u)\in C^{1,\frac{1}{2}}(Q_T).
\end{align}

Finally, let us check that the limit $u$ satisfies \ref{def:entr} in Definition \ref{def:entropy}. Here we need to introduce a piecewise constant inteporlation
of our data points $U^n_i$. We call
\begin{align*}
\bar{u}_{\Delta x}(x,t)=U_{i}^{n}\text{ for all $(x,t)\in[x_i,x_{i+1})\times[t_n,t_{n+1})$}.
\end{align*}
We do this since the discontinuous sign function $\eta_k'$ makes it difficult to work with the bilinear interpolant $\hat u_{\Delta x}$ in what follows. The
need for the piecewise linear interpolation was dictated by the condition \ref{def:spacetimereg}: continuity of the functions $A(\hat{u}_{\Delta x})$ were
needed to prove H\"{o}lder space-time regularity for the limit $A(u)$ (cf.~the proof of \eqref{kkjj}). To verify that the limit $u$ also satisfies
\ref{def:entr} the piecewise constant interpolation $\bar{u}_{\Delta x}$ suffices since, as we already have strong convergence for the piecewise linear
interpolation, strong convergence toward the same limit $u$ for the piecewise constant interpolation is ensured thanks to the fact that
$$\|\bar{u}_{\Delta x}(\cdot,t)- \hat{u}_{\Delta x}(\cdot,t)\|_{L^1(Q_T)}\leq c|U^n|_{BV(\Z)}\Delta x.$$ We now take a positive test function $\varphi\in
C_c^{\infty}(\R\times[0,T))$, and let $\varphi_i^n=\varphi(x_i,t_n)$. We multiply both sides of \eqref{entropcell} by $\varphi_i^n$, and sum over all $(i,n)$.
Using summation by parts, we obtain
\begin{align}\label{ex}
\begin{split}
\Delta x\Delta t\sum_{n=0}^{N-1}&\sum_{i\in\mathbb{Z}}\eta_{i}^{n}\frac{\varphi_{i}^{n+1}-\varphi_{i}^{n}}{\Delta t}\\
&+\Delta x\Delta t\sum_{n=0}^{N-1}\sum_{i\in\mathbb{Z}}Q_{i}^{n}D_{+}\varphi_i^n\\
&+\Delta x\Delta t\sum_{n=0}^{N-1}\sum_{i\in\mathbb{Z}}|A(U^n_i)-A(k)|D_{-}(D_{+}\varphi_i^n)\\
&+\Delta x\Delta t\sum_{n=0}^{N-1}\sum_{i\in\mathbb{Z}}\eta'_k(U^{n+1}_i)\Levy\langle U^{n}\rangle_i\varphi_{i}^{n}\\
&+\Delta x\sum_{i\in\mathbb{Z}}\varphi_{i}^{0}\eta_{i}^{0}\geq0.
\end{split}
\end{align}
A standard argument shows that all the local terms in the above expression converge to the ones appearing in the entropy inequality \ref{def:entr}, see
e.g.~\cite{Holden/Risebro,Evje/Karlsen}. Let us look the term containing the nonlocal operator $\Levy\langle\cdot\rangle$. We can rewrite it as
\begin{align*}
\begin{split}
\int_{\Delta t}^{T+\Delta t}\int_{\mathbb R}\eta'_{k}(\bar{u}_{\Delta x}) \Levy[\bar{u}_{\Delta x}]\bar{\varphi}\dx\dt+R,
\end{split}
\end{align*}
where $R\stackrel{\Delta x\rightarrow0}{\longrightarrow}0$ and $\bar\varphi$ is the piecewise constant interpolant of $\varphi_i^n$. Indeed, let us write
$\eta_k'(U^{n+1}_i)\Levy\langle U^{n}\rangle_i=\eta_k'(U^{n+1}_i)\Levy\langle U^n-U^{n+1}\rangle_i +\eta_k'(U^{n+1}_i)\Levy\langle U^{n+1}\rangle_i$. Note that
\begin{align*}
\begin{split}
\Delta x\Delta t\sum_{n=0}^{N-1}\sum_{i\in\mathbb{Z}} |\Levy\langle U^n-U^{n+1}\rangle_i|\varphi_{i}^{n}\leq\Delta
t\|\bar{\varphi}\|_{L^\infty(Q_T)}\sum_{n=0}^{N-1}\int_{\R}|\Levy[\bar{U}^n(x)-\bar{U}^{n+1}(x)]|\dx,
\end{split}
\end{align*}
where the last quantity vanishes as $\Delta x\rightarrow0$ by  $L^1$-Lipschitz continuity in time (cf. Lemma \ref{Liptime}, and also Lemmas \ref{stab} and
\ref{lemmanonloc}). Next,
\begin{align*}
\begin{split}
&\sum_{n=0}^{N}\sum_{i\in\mathbb{Z}}\eta'_{k}(U^{n+1}_i)
\Levy\langle{U}^{n+1}\rangle_i\varphi_i^{n}\\
&=\sum_{n=0}^{N}\sum_{i\in\mathbb{Z}}\eta'_{k}(U^{n+1}_i) \Levy\langle{U}^{n+1}\rangle_i(\varphi_i^{n}-\varphi_i^{n+1})
+\sum_{n=0}^{N}\sum_{i\in\mathbb{Z}}\eta'_{k}(U^{n+1}_i) \Levy\langle{U}^{n+1}\rangle_i\varphi_i^{n+1},
\end{split}
\end{align*}
where the first term on the right-hand side vanishes as $\Delta x\rightarrow0$ since there exists a constant $c_{\varphi}>0$ such that
$|\varphi_i^{n}-\varphi_i^{n+1}|\leq c_{\varphi}\Delta x$ for all $(i,n)$. To conclude, we prove that up to a subsequence and for a.e.~$k\in\mathbb{R}$,
\begin{align}\label{hhh}
\begin{split}
\int_{\Delta t}^{T+\Delta t}\int_{\mathbb R}\eta'_{k}(\bar{u}_{\Delta x}) \Levy[\bar{u}_{\Delta x}]\bar{\varphi}\dx\dt\stackrel{\Delta x\rightarrow
0}{\longrightarrow}\int_{Q_T}\eta'_k(u)\Levy[u]\varphi \dx\dt.
\end{split}
\end{align}
This is a consequence of the dominated convergence theorem since the left hand side integrand converges pointwise a.e.~to the right hand side integrand.
Indeed, first note that $\bar{\varphi}\rightarrow\varphi$ pointwise on $Q_T$, while a.e.~up to a subsequence, $\bar{u}_{\Delta x}\rightarrow u$ on $Q_T$. We
also have $\eta_k'(\bar{u}_{\Delta   x})\rightarrow\eta'_k(u)$ a.e.~in $Q_T$ since for a.e.~$k\in\mathbb{R}$ the measure of $\{(x,t)\in Q_T:u(x,t)=k\}$ is zero
and $\eta'_k$ is continuous on $\mathbb{R}\backslash \{k\}$. Finally if the (compact) support of $\varphi$ is containd in $[-R,R]\times[0,T]$, $R>0$, then a
trivial extension of Lemma \ref{lemmanonloc} implies that
\begin{align*}
\begin{split}
\int_{[-R,R]\times[0,T]}|\Levy[\bar{u}_{\Delta x}-u]|\dx\dt\leq c_\lambda C\int_{0}^{T}\|\bar{u}_{\Delta x}-u\|_{L^1(-R,R)}^{1-\lambda} |\bar{u}_{\Delta
x}-u|_{BV(-R,R)}^{\lambda},
\end{split}
\end{align*}
where the last quantity vanishes as $\Delta x\rightarrow0$ since $\bar{u}_{\Delta x}\rightarrow u$ in $L^1_{loc}(Q_T)$. Then $\Levy[\bar{u}_{\Delta
  x}]\rightarrow \Levy[u]$ a.e.~in $[-R,R]\times(0,T)$ up to a
subsequence. Convergence for all $k\in\mathbb{R}$ can be proved along the lines of \cite[Lemmas 4.3 and 4.4]{Karlsen:2003lz}.
\end{proof}

\subsection{Remarks on the LDG method}\label{sec:remarksldg}
The derivation of the LDG method in the piecewise constant case is not as straightforward as the one for the DDG method. Indeed, the numerical fluxes
introduced in \eqref{Lflux} depend on the choice of the function $c_{12}$, and computations cannot be performed until this function has been defined. Our aim
now is to show that the LDG method reduces to a numerical method similar to \eqref{Ds3} for a suitable choice of the function $c_{12}$.

Let us for the time being ignore the nonlinear convection and
 fractional diffusion terms and focus on the problem
\begin{align*}
\left\{
\begin{array}{ll}
u_t-\partial_x\sqrt{a(u)}q=0,\\
q-\partial_xg(u)=0,\\
u(x,0)=u_{0}(x).
\end{array}
\right.
\end{align*}
The LDG method \eqref{systdisc} then takes the form
\begin{align}\label{systpwc}
\begin{cases}
\int_{I_{i}}\tilde{u}_t+\hat{h}_u(\tilde{\mathbf{w}}_{i+1})
-\hat{h}_u(\tilde{\mathbf{w}}_{i})=0,\\[0.2cm]
\int_{I_{i}}\tilde{q}+\hat{h}_q(\tilde{u}_{i+1})-\hat{h}_q(\tilde{u}_{i})=0,
\end{cases}
\end{align}
where $\tilde{u}(x,t)=\sum_{i\in\mathbb{Z}}U_{i}(t)\mathbf{1}_{I_i}(x)$, $\tilde{q}(x,t)=\sum_{i\in\mathbb{Z}}Q_{i}(t)\mathbf{1}_{I_i}(x)$, and the fluxes
$(\hat h_u,\hat h_q)$ are defined in \eqref{Lflux}. Let us insert $\tilde{u}$ and $\tilde{q}$ into the system \eqref{systpwc}, and use the flux \eqref{Lflux}
to get
\begin{align}\label{ggg}
\begin{cases}
\frac{d}{dt}U_{i}\Delta x-\frac{g(U_{i+1})-g(U_{i})}{U_{i+1}
-U_{i}}\frac{Q_{i+1}+Q_{i}}{2}-c_{12}(Q_{i+1}-Q_{i})\\[0.2cm]
\qquad\qquad+\frac{g(U_{i})-g(U_{i-1})}{U_{i}-U_{i-1}}\frac{Q_{i}
+Q_{i-1}}{2}+c_{12}(Q_{i}-Q_{i-1})=0,\\[0.2cm]
Q_{i}\Delta x-\frac{g(U_{i+1})+g(U_{i})}{2}+c_{12}(U_{i+1}-U_{i})\\[0.2cm]
\qquad\qquad+\frac{g(U_{i})+g(U_{i-1})}{2}-c_{12}(U_{i}-U_{i-1})=0.
\end{cases}
\end{align}
Let us choose the function $c_{12}$ to be
\begin{align}\label{c12}
c_{12}(U_{i},U_{i-1})=\frac{1}{2}\frac{g(U_{i})-g(U_{i-1})}{U_{i}-U_{i-1}}.
\end{align}
Inserting \eqref{c12} into \eqref{ggg} then leads to
\begin{align*}
\begin{cases}
\frac{d}{dt}U_{i}\Delta x-\frac{g(U_{i+1})-g(U_{i})}{U_{i+1}-U_{i}}Q_{i+1}
+\frac{g(U_{i})-g(U_{i-1})}{U_{i}-U_{i-1}}Q_{i}=0,\\[0.2cm]
Q_{i}=\frac{g(U_{i})-g(U_{i-1})}{\Delta x},
\end{cases}
\end{align*}
or
\begin{align*}
\frac{d}{dt}U_{i}\Delta x-\frac{1}{\Delta x}\frac{(g(U_{i+1})-g(U_{i}))^2}{U_{i+1}-U_{i}} +\frac{1}{\Delta x}\frac{(g(U_{i})-g(U_{i-1}))^2}{U_{i}-U_{i-1}}=0.
\end{align*}

For the full equation \eqref{1}, this choice of $c_{12}$ along with a forward difference approximation in time, lead to the following piecewise constant LDG
approximation:
\begin{align}\label{Ls1}
\begin{split}
&\frac{U_i^{n+1}-U_i^{n}}{\Delta t}+\frac{\hat{f}(U_i^n,U_{i+1}^n) -\hat{f}(U_{i-1}^n,U_{i}^n)}{\Delta x}\\&-\frac{1}{\Delta
x^2}\frac{(g(U^n_{i+1})-g(U^n_{i}))^2}{U_{i+1}^n-U_{i}^n} +\frac{1}{\Delta x^2}\frac{(g(U^n_{i})-g(U^n_{i-1}))^2}{U_{i}^n-U_{i-1}^n} =\frac{1}{\Delta
x}\sum_{j\in\Z}G^i_jU_j^n.
\end{split}
\end{align}
\begin{remark}
We do not prove convergence for the numerical method \eqref{Ls1}. However, we note that
\begin{align*}
\begin{split}
\left(\frac{dg}{du}\right)^2=\frac{dA}{du},
\end{split}
\end{align*}
since $g=\int^{u}\sqrt{a}$ and $A=\int^{u}a$. Roughly speaking this means that
\begin{align*}
\begin{split}
\frac{(g(U^n_{i+1})-g(U^n_{i}))^2}{U_{i+1}-U_{i}}\approx A(U^n_{i+1})-A(U^n_{i}),
\end{split}
\end{align*}
and hence that \eqref{Ds3} and \eqref{Ls1} are closely related. Experiments indicates that the two methods produce similar solutions (cf.~Figure
\ref{figure4}).
\end{remark}

\section{Numerical experiments}

We conclude this paper by presenting some experimental results obtained using the fully explicit (piecewise constant) numerical methods \eqref{scheme} and
\eqref{Ls1}, and the DDG method \eqref{Ds2} with fully explicit third order Runge-Kutta time discretization and piecewise constant, linear, and quadratic
elements. In the computations we have imposed a zero Dirichlet boundary condition on the whole exterior domain $\{|x|>1\}$. In all the plots, the dotted line
represents the initial datum while the solid one (or the dashed-dotted one in Figure \ref{figure4}) the numerical solution at $t=T$.

\begin{remark}
The operator $\Levy[\hat u]$ requires the evaluation of the discrete solution $\hat u$ on the whole real axis, thus making necessary the use of some
localization procedure. In our numerical experiments we have confined the nonlocal operator $\Levy[\cdot]$ to the domain $\Omega=\{|x|\leq 1\}$. That is to
say, for each grid point $(x_i,t_n)\in\Omega\times (0,T)$ we have computed the value of $\hat u$ at time $t_{n+1}$ by using only the values $\hat u(x_i,t_n)$
with $x_i\in\Omega$.
\end{remark}

We consider two different sets of data taken from \cite{Evje/Karlsen}. In Example 1 we take
\begin{align}\label{ex1}\tag{\textbf{Ex}.1}
\begin{split}
f_1(u)&=u^{2},\\
a_1(u)&=
\begin{cases}
0&\text{for} \ u\leq0.5\\
2.5u-1.25&\text{for} \ 0.5<u\leq0.6\\
0.25&\text{for} \ u>0.6,
\end{cases}\\
u_{0,1}(x)&=
\begin{cases}
0&\text{for} \ x\leq-0.5\\
5x+2.5&\text{for} \ -0.5<x\leq-0.3\\
1&\text{for} \  -0.3<x\leq0.3\\
2.5-5x&\text{for} \ 0.3<x\leq0.5\\
0&\text{for} \ x>0.5.
\end{cases}
\end{split}
\end{align}
In Example 2 we choose
\begin{align}\label{ex2}\tag{\textbf{Ex}.2}
\begin{split}
f_2&=\frac{1}{4}\,f_1,\\
a_2&=4\,a_1,\\
u_{0,2}(x)&=
\begin{cases}
1&\text{for} \ x\leq-0.4\\
-2.5x&\text{for} \ -0.4<x\leq0\\
0&\text{for} \ x>0.
\end{cases}
\end{split}
\end{align}
Furthermore, in Example 3 we use
\begin{align}\label{ex3}\tag{\textbf{Ex}.3}
\begin{split}
f_3(u)&=u,\\
a_3(u)&=0.1,\\
u_{0,3}(x)&=\exp{\left(-\left(\frac{x}{0.1}\right)^2\right)}.
\end{split}
\end{align}

The numerical results are presented in Figure \ref{fig1}, \ref{fig2}, \ref{figure4}, and \ref{fig_add}. The results confirm what we expected: the solutions of
the initial value problem \eqref{1} can develop shocks in finite time (this feature has been proved in \cite{Alibaud/Droniou/Vovelle} for the case $a=0$). In
Figure \ref{fig1} and \ref{fig2} you can see how the presence of the fractional diffusion $\Levy$ influences the shock's size and speed. In Figure
\ref{fig_add} you can see how the accuracy of DDG method \eqref{Ds2} improves when high-order polynomials are used ($k=0,1,2$).

\begin{figure}[h]
\subfigure[$u_t+f(u)_x=(a(u)u_x)_x$]{
\includegraphics[scale=0.4]{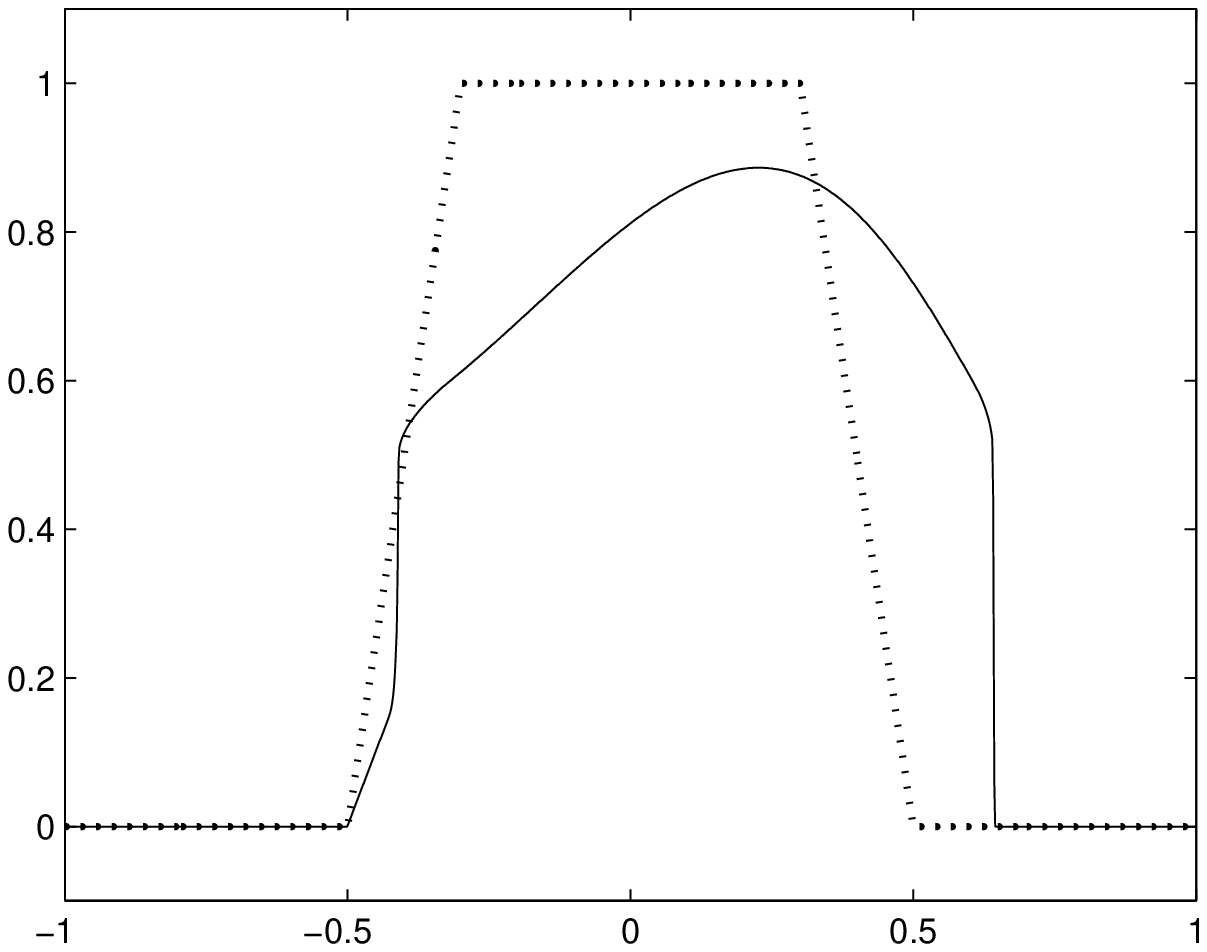}}
\subfigure[Equation \eqref{1} with $\lambda=0.5$]{
\includegraphics[scale=0.4]{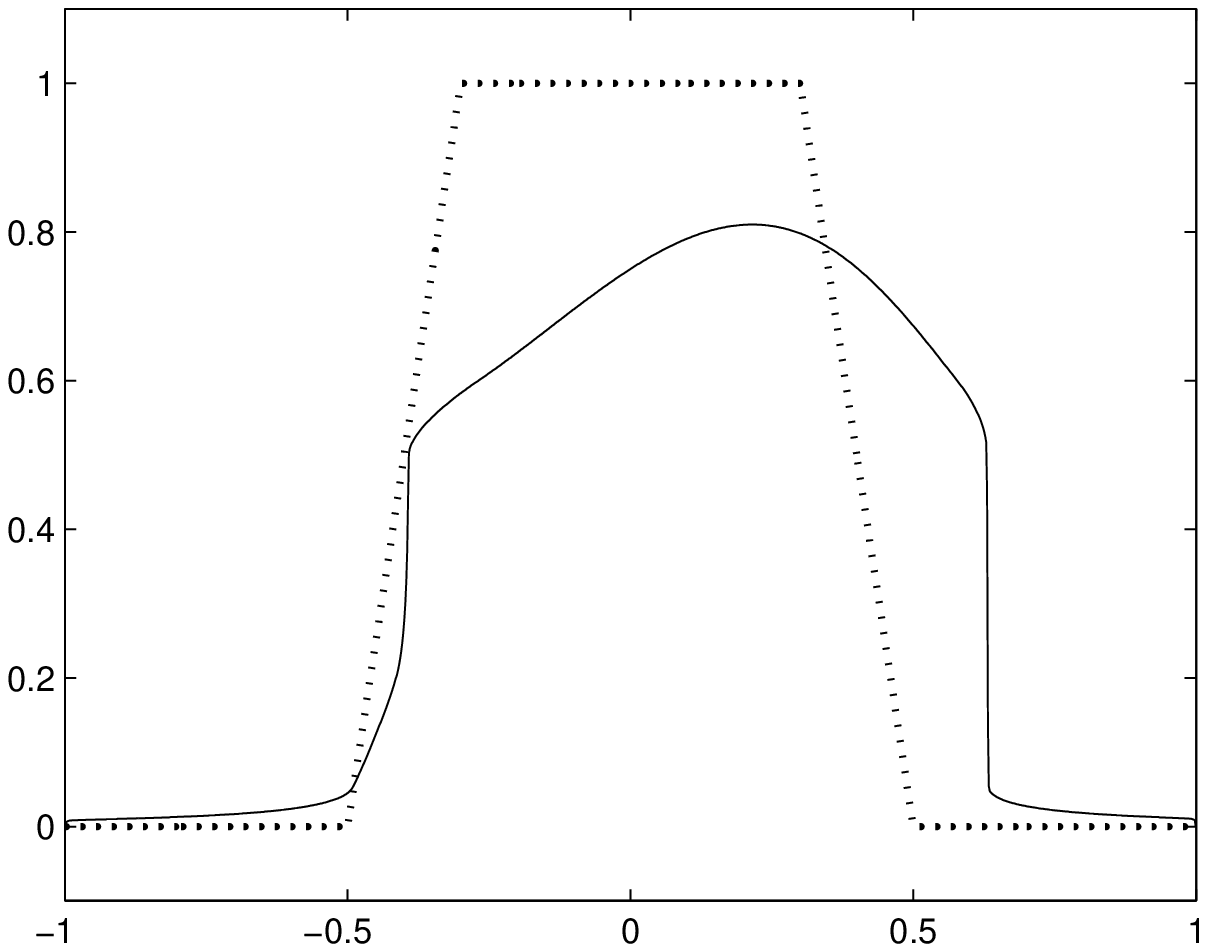}}
\caption{(\ref{ex1}): $T=0.15$ and $\Delta x=1/640$.}\label{fig1}
\end{figure}
\begin{figure}[h]
\subfigure[$u_t+f(u)_x=(a(u)u_x)_x$]{
\includegraphics[scale=0.4]{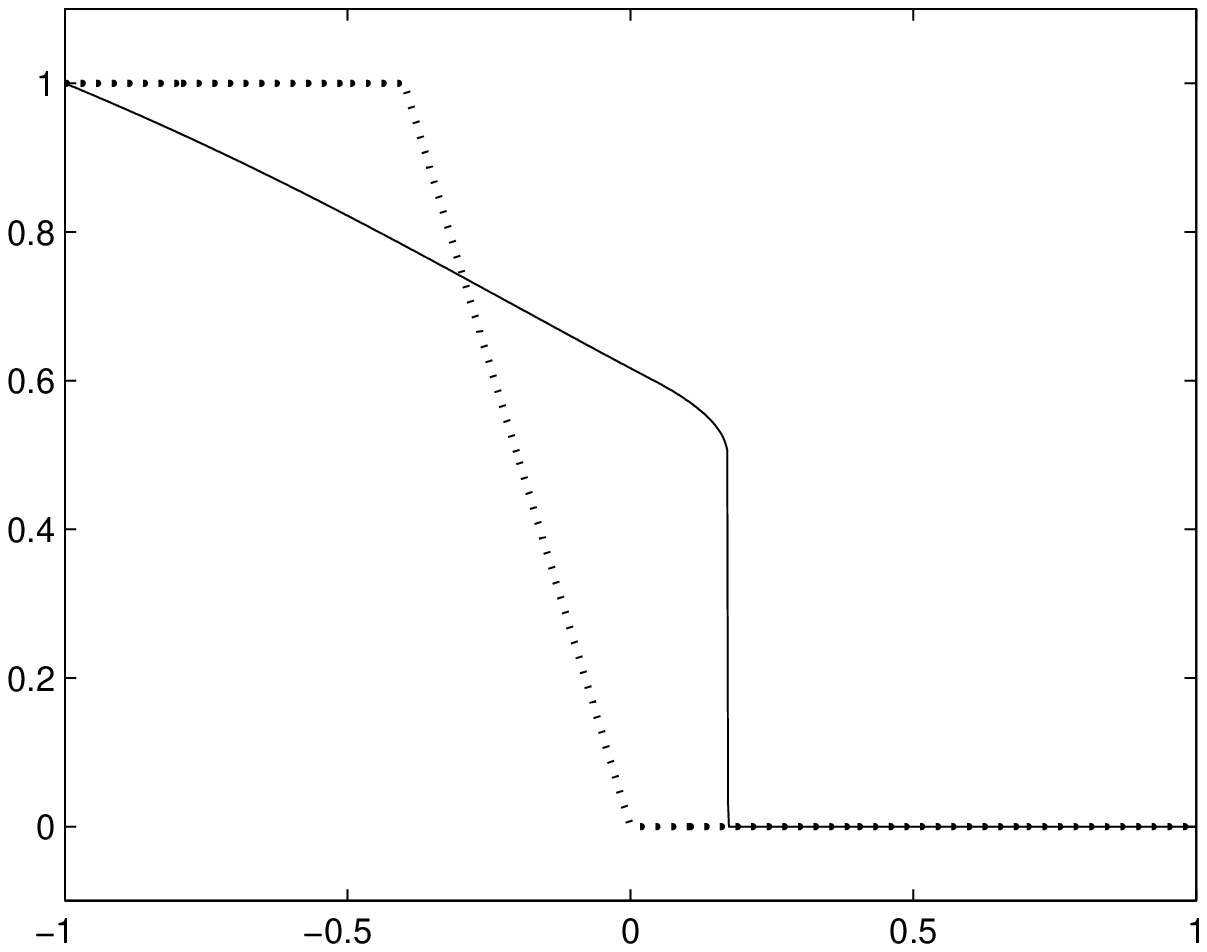}}
\subfigure[Equation \eqref{1} with $\lambda=0.5$]{
\includegraphics[scale=0.4]{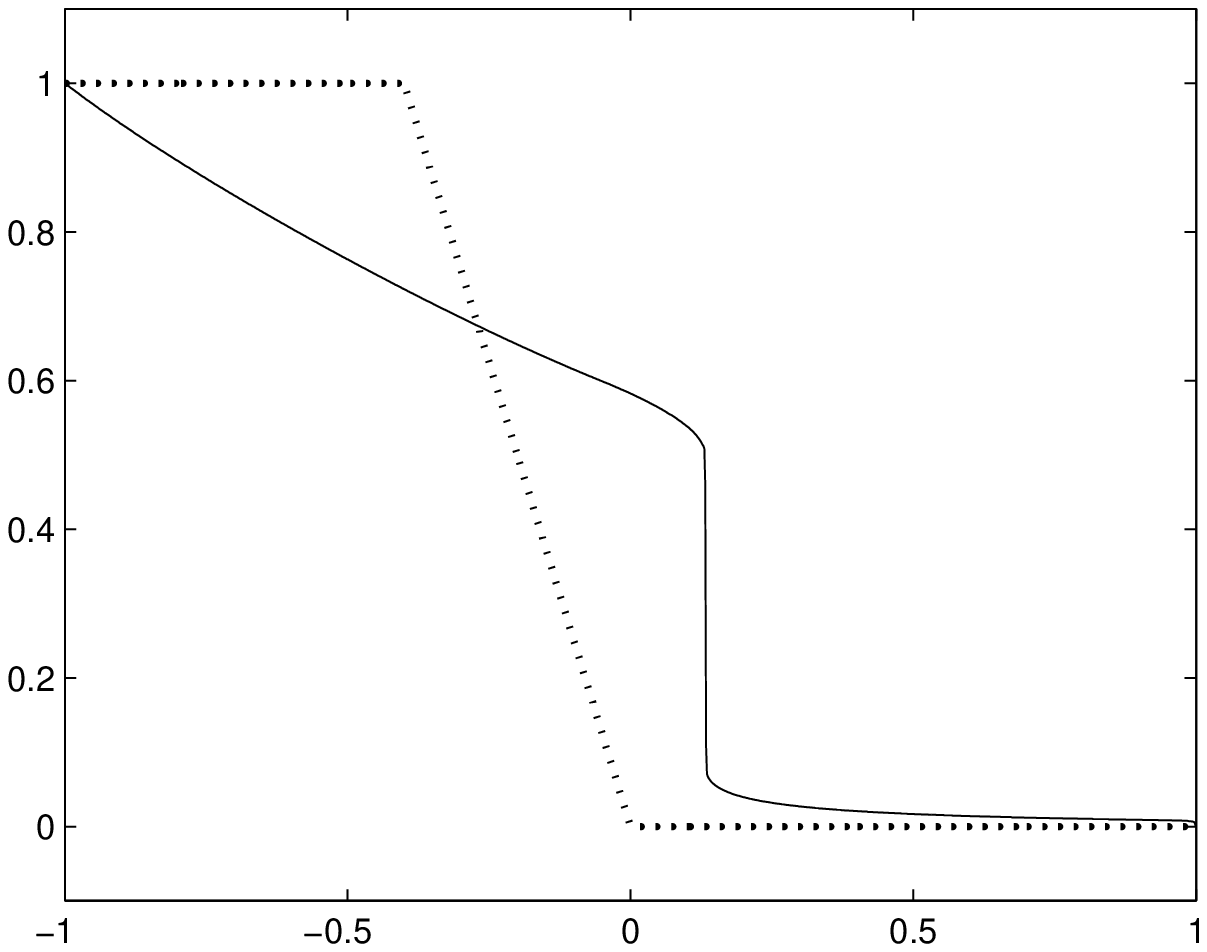}}
\caption{(\ref{ex2}): $T=0.25$ and $\Delta x=1/640$.}\label{fig2}
\end{figure}
\begin{figure}[h]
\subfigure[$T=0.0625$]{
\includegraphics[scale=0.4]{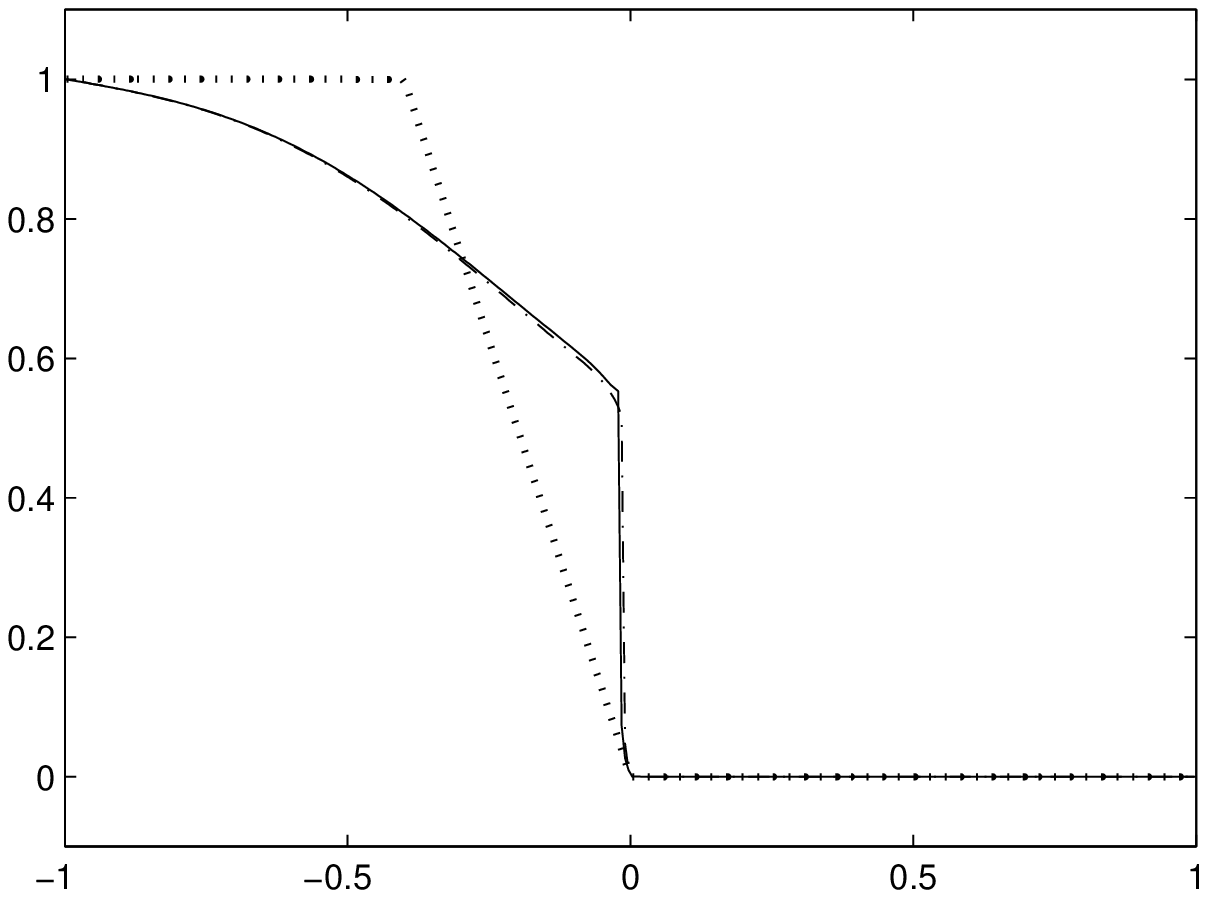}}
\subfigure[$T=1$]{
\includegraphics[scale=0.4]{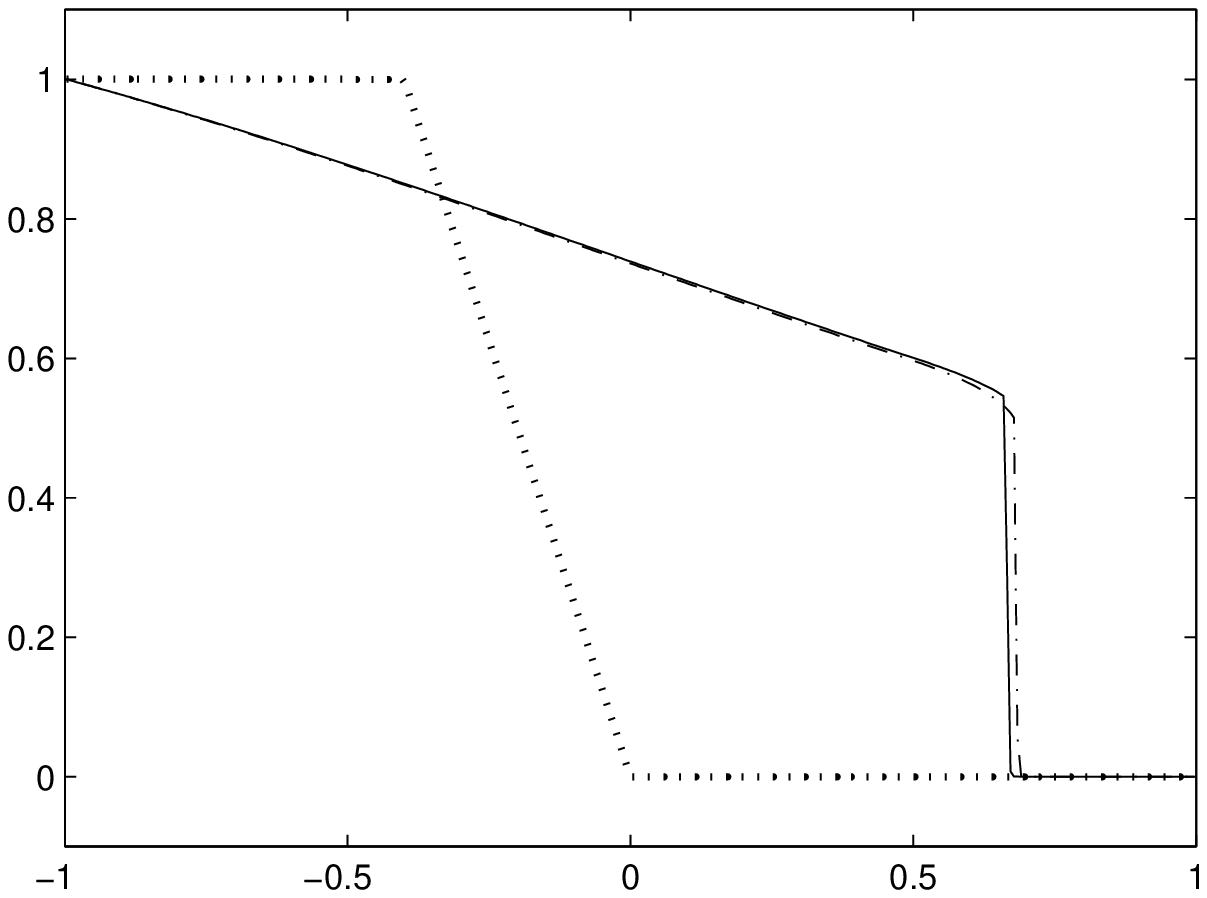}}
\caption{(\ref{ex2}): solutions of $u_t+f(u)_x=(a(u)u_x)_x$ at different times using methods \eqref{scheme} and \eqref{Ls1} ($\Delta x=1/160$).}\label{figure4}
\end{figure}

\begin{figure}[ht]
\subfigure[Piecewise constant ($k=0$) with $\Delta x=1/20$]{
\includegraphics[scale=0.4]{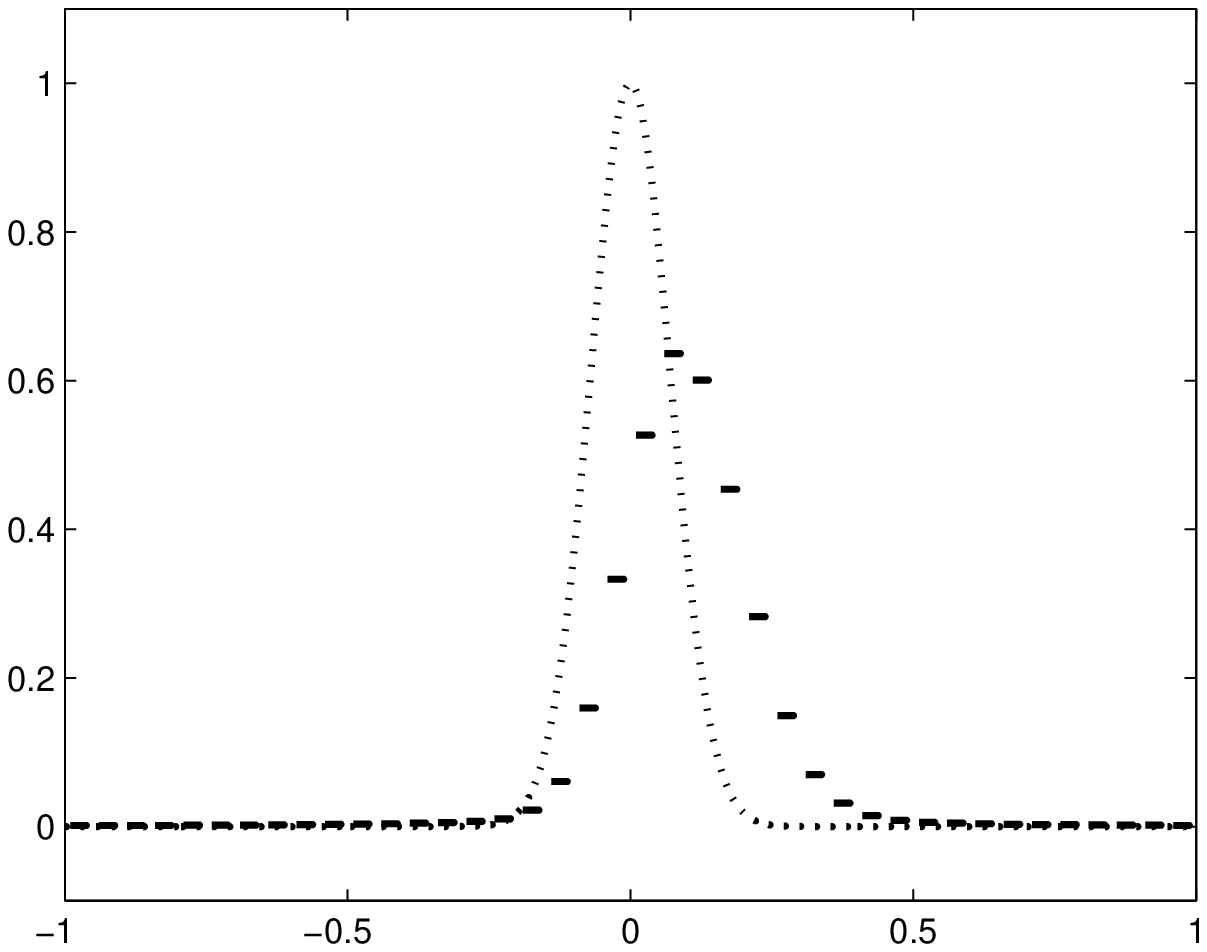}}
\subfigure[Piecewise linear ($k=1$) with $\Delta x=1/20$]{
\includegraphics[scale=0.4]{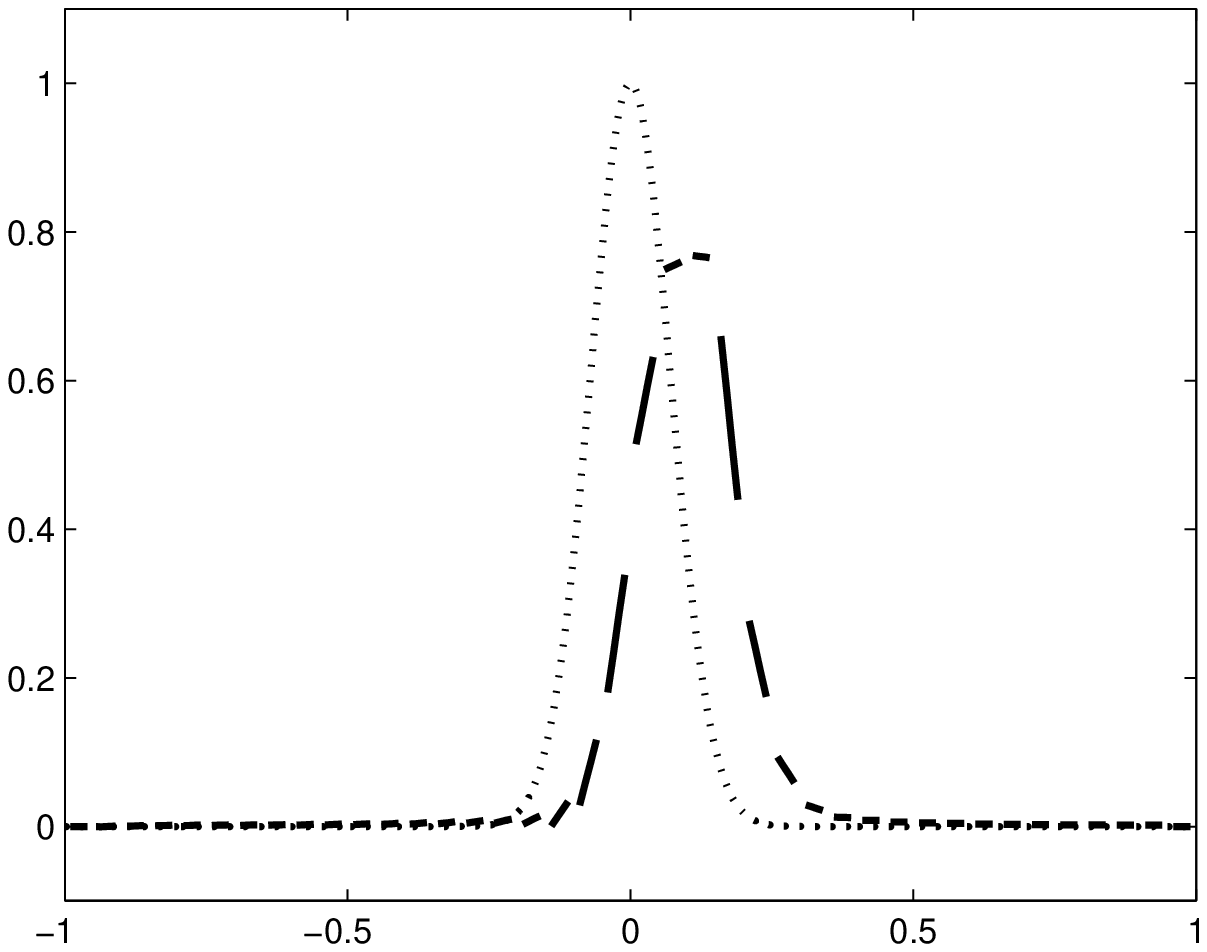}}\\
\subfigure[Piecewise quadratic ($k=2$) with $\Delta x=1/20$]{
\includegraphics[scale=0.4]{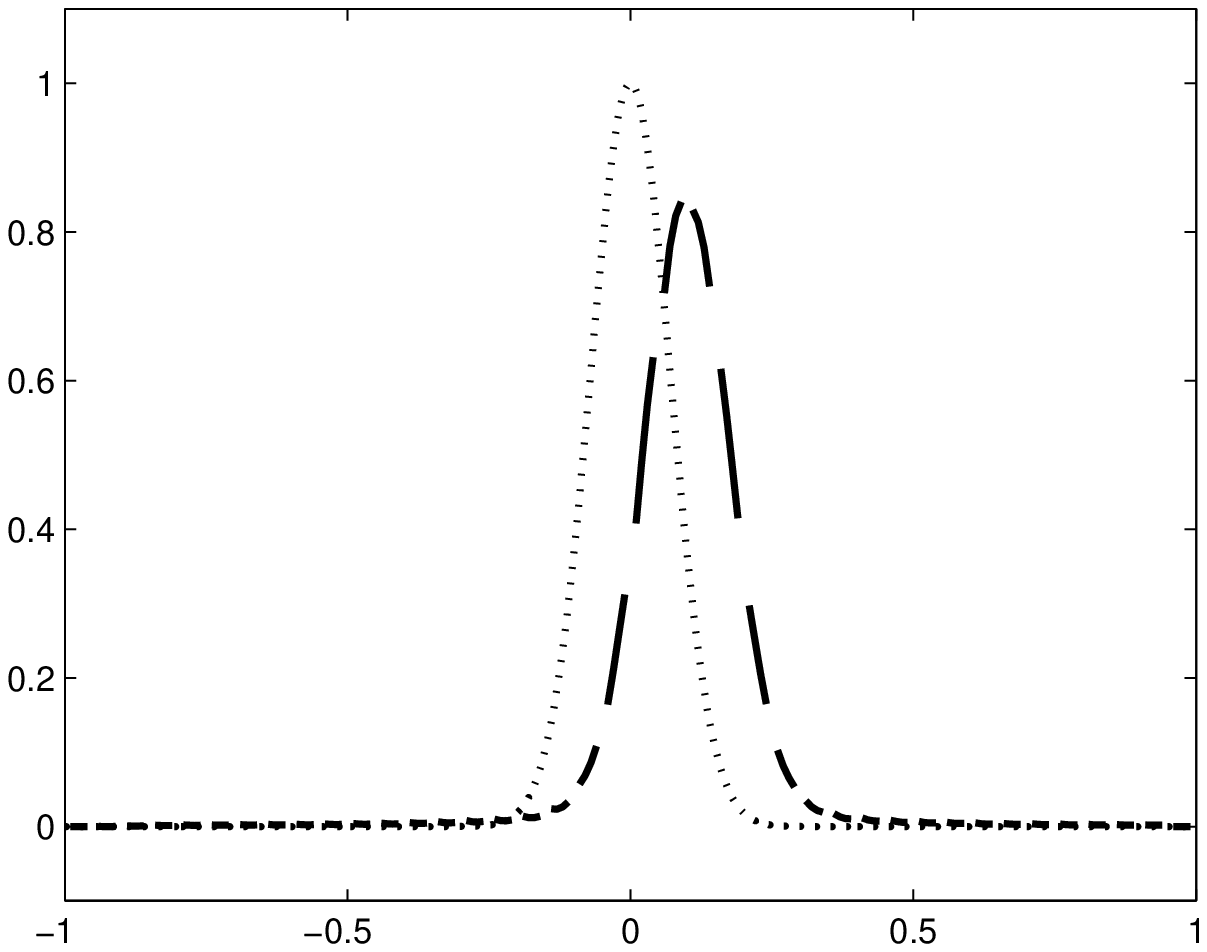}}
\subfigure[Solution computed using $\Delta x=1/640$]{
\includegraphics[scale=0.4]{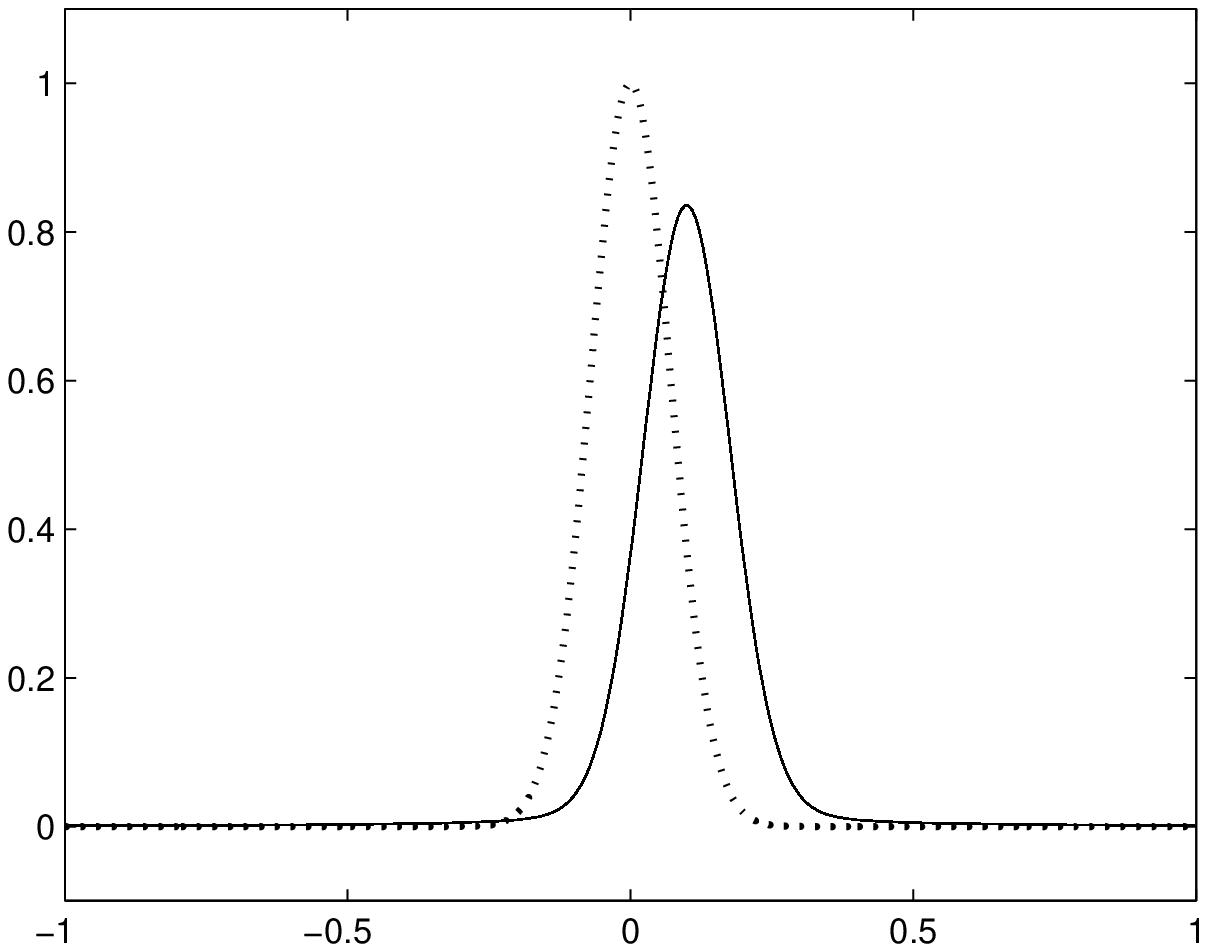}}
\caption{(\ref{ex3}): solutions at $T=0.1$ using $k=0,1,2$.}\label{fig_add}
\end{figure}

In Figure \ref{figure4}, the dashed-dotted curve represents method \eqref{scheme}, while the solid one represents method \eqref{Ls1}. The two numerical
solutions stay close, and numerical convergence has been observed for finer grids. Note that here we have set $b=0$ (no fractional diffusion) in order to
stress the differences between the two methods.

The numerical rate of convergence for the solutions in Figure \ref{fig1} (b), \ref{fig2} (b), and \ref{fig_add} (b) are presented in Table \ref{table1}. We
have measured the $L^p$-error
$$E_{\Delta x,p}=\|\hat{u}_{\Delta
  x}(\cdot,T)-\hat{u}_{e}(\cdot,T)\|^p_{L^{p}(\mathbb{R})},$$
where $\hat{u}_{e}$ is the numerical solution which has been computed using a very fine grid ($\Delta x=1/640$), the relative error
$$R_{\Delta
x,p}=\left(\frac{1}{\|\hat{u}_{e}(\cdot,T)\|^p_{L^{p}(\mathbb{R})}}\right)\,E_{\Delta x,p},$$ and the approximate rate of convergence $$\alpha_{\Delta
x,p}=\left(\frac{1}{\log 2}\right)\Big(\log E_{\Delta x,p}-\log E_{\Delta x/2,p}\Big).$$

\begin{table}[h]
\caption{Error, relative error, and numerical rate of convergence for
  the solutions in Figure \ref{fig1} (b), \ref{fig2} (b), and \ref{fig_add} (b).}\label{table1}
\begin{tabular}{c| c c c | c c c|c c c }
&\multicolumn{3}{c|}{Figure \ref{fig1} (b)}&\multicolumn{3}{c}{Figure
  \ref{fig2} (b)} & \multicolumn{3}{c}{Figure \ref{fig_add} (b)}\\
$\Delta x$ & $E_{\Delta x,1}$ & $R_{\Delta x,1}$ & $\alpha_{\Delta x,1}$ & $E_{\Delta x,1}$ & $R_{\Delta x,1}$ & $\alpha_{\Delta x,1}$& $E_{\Delta x,2}$ &
$R_{\Delta x,2}$ & $\alpha_{\Delta x,2}$
\\
    \hline
$\mathbf{1/10}$    & 0.0706 & 0.0942 & 0.97  &  0.0474   & 0.0550   & 0.86& 0.009000 & 0.093595 & 2.00 \\
$\mathbf{1/20}$    & 0.0361 & 0.0482 & 0.92 &  0.0261   & 0.0302   & 0.49& 0.002300 & 0.023493 & 1.85 \\
$\mathbf{1/40}$    & 0.0191 & 0.0255 & 0.57 &  0.0186   & 0.0216   & 0.52 & 0.000626 & 0.006518 & 1.54  \\
$\mathbf{1/80}$    & 0.0128 & 0.0171 & 0.60 &  0.0130   & 0.0150   & 0.42 & 0.000216 & 0.002248 & 1.10 \\
$\mathbf{1/160}$   & 0.0084 & 0.0113 & 0.76  &  0.0097   & 0.0112   & 0.77 & 0.000101 & 0.001052 & 1.04  \\
$\mathbf{1/320}$   & 0.0050 & 0.0066 & -      &  0.0057   & 0.0066   & -     & 0.000049 & 0.000510 & -         \\
    \hline
\end{tabular}
\end{table}

Our simulations seem to indicate numerical convergence of order less than one for the solutions depicted in Figure \ref{fig1} (b) and \ref{fig2} (b) (nonlinear
equations and piecewise constant elements), and numerical convergence of order higher than one for the solution depicted in Figure \ref{fig_add} (b) (linear
equation and piecewise linear elements). In the last case we do not seem to reach the expected value 2 (cf.~the statement of Theorem \ref{DDGlin}). This
deterioration of the numerical order of convergence for high-order polynomials has already been observed by the authors in \cite{Cifani/ERJ/KHK}. The reasons
behind this deterioration are still not clear.

Finally, let us remind the reader that no general results concerning the rate of convergence of numerical methods for nonlinear equations like \eqref{1} have
been produced so far. For more details, cf.~\cite{Chen/Karlsen}.


\appendix
\section{Technical lemmas}
In this appendix we state some technical results from \cite{Cifani/ERJ/KHK} that are needed in this paper. All proofs can be found in \cite{Cifani/ERJ/KHK}.
\begin{lemma}\label{lemmanonloc}
Let $\varphi,\phi\in L^1(\R)\cap BV(\R)$. Then there exists $C>0$ such that
\begin{align}
&\int_{\R}|\Levy[\varphi]|\leq c_\lambda C
\|\varphi\|_{L^1(\R)}^{1-\lambda}|\varphi|_{BV(\R)}^{\lambda},\label{kkk1}\\
&\int_\R\phi\Levy[\varphi]=\int_\R\varphi\Levy[\phi],\label{kkk2}\\
&\int_\R\varphi\Levy[\varphi]=-\frac{c_\lambda}{2}\int_\R\int_\R\frac{(\varphi(z)-\varphi(x))^2}{|z-x|^{1+\lambda}}\dz\dx.\label{kkk3}
\end{align}
Moreover, the last two identities also hold for all functions $\phi,\varphi\in H^{\lambda/2}(\R)$.
\end{lemma}

To prove inequality \eqref{kkk1} one can split the nonlocal operator $\Levy[\cdot]$, using an auxiliary parameter $\epsilon>0$, into the sum of
$\Levy_\epsilon[\cdot]$, the operator containing the singularity, and $\Levy^\epsilon[\cdot]$, the remaining part of the original operator. The operator
$\Levy_\epsilon[\cdot]$ can then be treated using the control on the bounded variation, while the control on the $L^1$-norm is needed for the operator
$\Levy^\epsilon[\cdot]$. To obtain exactly estimate \eqref{kkk1} the optimal value of $\epsilon$ must be chosen. The proof of \eqref{kkk2} - and thus of
\eqref{kkk3} - is essentially a change of variables.

\begin{lemma}\label{matG}
For all $(i,j)\in\Z\times\Z$,
\begin{align*}
\begin{split}
\sum_{k\in\Z}|G^i_k|<\infty,\ \sum_{k\in\Z}G^i_k=0,\ G^i_j=G_i^j \text{ and $G^{i+1}_{j+1}=G_j^i$.}
\end{split}
\end{align*}
Moreover, $G^i_j\geq0$ whenever $i\neq j$, while
\begin{align}
G^i_i=-c_{\lambda}\left(\int_{|z|<1}\frac{dz}{|z|^{\lambda}}+\int_{|z|>1}\frac{dz}{|z|^{1+\lambda}}\right)\Delta x^{1-\lambda}\leq0.\label{kkk4}
\end{align}
\end{lemma}

Lemma \ref{matG} is essentially a consequence of the form of the operator $\Levy[\cdot]$ itself, and properties \eqref{kkk1} and \eqref{kkk2}. Property
\eqref{kkk4} comes from a precise evaluation of the integral $G_i^i$.

\begin{lemma} \label{Hs} If $\phi\in V^k\cap L^2(\R)$, then $\phi\in
  H^{\frac\lambda2}(\R)$ for all $\lambda\in(0,1)$, and
\begin{equation}\label{kkkk1}
\begin{split}
\|\phi\|_{H^{\frac\lambda2}(\R)}^2\leq\frac C{\Delta x}\|\phi\|_{L^2(\R)}^2.
\end{split}
\end{equation}
\end{lemma}

Lemma \ref{Hs} is essentially a consequence of the fact that $\phi$ is a piecewise polynomial. The control on the $L^2$-norm together with the piecewise
structure of $\phi$ ensure that its quadratic variation is bounded. Then, the finite quadratic variation plus the fact that $\phi$ is differentiable inside
each interval $I_i$ return \eqref{kkkk1}.

%

%
%

\end{document}